\documentclass[reqno]{amsart}

\usepackage{amssymb,amsfonts,amsmath,mathdots,amsthm}
\usepackage{float}

\usepackage{tikz-cd}
\usepackage{lmodern}
\usepackage{mathrsfs}
\usepackage[sc]{mathpazo}
\DeclareMathAlphabet{\mathfrak}{U}{euf}{m}{n}
\usepackage{mathtools}  %xhook, \ev
\usepackage{calligra}
\usepackage{xcolor}
\usepackage[hidelinks]{hyperref}
\usepackage{accents}
\usepackage{subcaption}
\usepackage{dashbox}
\usepackage{enumitem}
\usepackage{tabu}

\usepackage{hyperref}

\hypersetup{colorlinks=true,
	linkcolor=magenta,
	citecolor=magenta}

\setlist[itemize]{leftmargin=*}

\newcommand{\genu}{\alpha}
\newcommand{\genw}{u}
\newcommand{\genc}{\gamma}
\newcommand{\gene}{c}
\newcommand{\geneidem}{e}
\newcommand{\gend}{\beta}

\usepackage{relsize}
\usepackage{etoolbox,fancyhdr}
\usepackage{indentfirst}  

\usetikzlibrary{arrows,arrows.meta,shapes.misc,decorations.markings,positioning,patterns}

\let\oldproofname=\proofname
\renewcommand{\proofname}{\bfseries\textup{\oldproofname}}

\mathtoolsset{showonlyrefs} %Only tag reffered equations

\DeclareMathOperator{\Z}{\mathbb{Z}}
\DeclareMathOperator{\C}{\mathbb{C}}
\DeclareMathOperator{\Q}{\mathbb{Q}}
\DeclareMathOperator{\F}{\mathbb{F}}
\DeclareMathOperator{\R}{\mathbb{R}}

\DeclareMathOperator{\Tr}{Tr}
\DeclareMathOperator{\Res}{Res}
\DeclareMathOperator{\dom}{dom}

\DeclareMathOperator{\Hom}{Hom}

\DeclareMathOperator{\limit}{\lim}
\DeclareMathOperator{\colimit}{\textup{colim}}

\theoremstyle{plain}

\newtheorem{prop}{Proposition}[section]

\newtheorem{cor}[prop]{Corollary}

\makeatother
\title{$C_2$ equivariant characteristic classes over the rational Burnside ring}
\author{Nick Georgakopoulos}

\begin{document}

	\begin{abstract}We give minimal presentations for the $RO(C_2)$-graded Bredon cohomology of the equivariant classifying spaces $B_{C_2}U(n), B_{C_2}SO(n)$ and $B_{C_2}Sp(n)$ with coefficients in the rational Burnside Green functor $A_{\Q}$. This results in an efficient description of rational $C_2$ equivariant Chern, Pontryagin and symplectic characteristic classes. These classes are then related to each other using the inclusions of maximal tori.
	\end{abstract}

	\maketitle{}
	\tableofcontents
	\section{Introduction}\label{Intro}
	
	Characteristic classes are classical and invaluable tools for understanding and distinguishing bundles over spaces.
	If we have a compact Lie group $G$ acting on a space $X$, there is a corresponding theory of $G$-equivariant bundles and $G$-equivariant characteristic classes.

	May proves in \cite{May87} that when Borel cohomology $$H^*_{G,Borel}(X)=H^*(X\times_GEG)$$ is used, the theory of Borel equivariant characteristic classes reduces to the non\-equivariant one, in the sense that $H^*_{G,Borel}(B_GL)=H^*(BG)\otimes H^*(BL)$ for any compact Lie group $L$ (which can be $L=U(n), SO(n), Sp(n)$ and so on).
	
	Equivariant characteristic classes in genuine (Bredon) equivariant cohomology are much less understood, owing to the significant complexity involved in computing it. 
	
	Recall that for a $G$-space $X$, unreduced $G$-equivariant (Bredon) cohomology $H^{\bigstar}_G(X)$ is not just a ring, but a Green functor: for every orbit $G/H$ we have a ring $H^{\bigstar}_G(X)(G/H)$ with an action from the Weyl group $W_GH=N_GH/H$ (where $N_GH$ is the normalizer of $H$ in $G$) and these rings are related to each other via restriction and transfer maps satisfying certain axioms. In more detail, for any subgroup inclusion $K\subseteq H$ we have a corresponding restriction and transfer maps: \begin{gather}
	\Res^H_K:H^{\bigstar}_G(X)(G/H)\to H^{\bigstar}_G(X)(G/K)\\
	\Tr^H_K:H^{\bigstar}_G(X)(G/K)\to H^{\bigstar}_G(X)(G/H)
	\end{gather}
	Moreover, the index $\bigstar$ is not just an integer, but an element of the real representation ring $RO(G)$. The coefficients used in $RO(G)$-graded cohomology are also Green functors and the initial ring $\Z$ is supplanted by the initial Burnside Green functor $A_{\Z}$. So $H^{\bigstar}_G(X)$ is by definition $H^{\bigstar}_G(X;A_{\Z})$ and we can more generally consider $H^{\bigstar}_G(X;R)$ for a $G$-Green functor $R$.
	
	Computing the coefficients of $RO(G)$-graded cohomology, namely the Green functor $H^{\bigstar}_G(*;A_{\Z})$, is a non-trivial undertaking on its own. The reader can consult \cite{Lew88} for the rather complicated answer when $G=C_p$ is the  cyclic group of prime order $p$. Even when we replace the coefficients $A_{\Z}$ by the constant Green functors corresponding to trivial $G$-modules $\Z$ and $\F_2$, the computations remain quite involved (see \cite{Geo19} and \cite{BC4S2} for the case of $G=C_4$).
	
	For characteristic classes, we further need to compute the $RO(G)$-graded cohomology of equivariant classifying spaces such as $B_GU(n), B_GSO(n)$ and $B_GSp(n)$. Such calculations for $n\le 3$, $G=C_2$ and using $A_{\Z}$ coefficients are performed in \cite{Shu14}, \cite{Cho18}. For $n=1$, $G=C_2$ and using constant $\F_2$ coefficients, the cohomology of $B_{C_2}O(1)=B_{C_2}\Sigma_2$ is the test module used in the determination of the dual Steenrod algebra (\cite{HK96}) and equivariant Dyer-Lashof operations (\cite{Wil19}). The same calculation for $G=C_4$ is significantly more complicated (\cite{BC4S2}).
		
	A way to simplify the algebra involved is to use coefficients in the rational Burnside Green functor $A_{\Q}$. Indeed, a result by Greenlees-May reduces the computation of the $RO(G)$-graded cohomology of a space $X$ in $A_{\Q}$ coefficients to nonequivariant rational cohomology of the fixed points $X^H$ where $H$ ranges over the subgroups of $G$ (\cite{GM95}). This allows us to compute explicit descriptions of the Green functors $H^{\bigstar}_G(B_GU(n);A_{\Q})$ , $ H^{\bigstar}_G(B_GSO(n);A_{\Q})$, $H^{\bigstar}_G(B_GSp(n);A_{\Q})$ and so on.
	
	However, those explicit descriptions are rather inefficient: For $G=C_2$, the ring $H^{\bigstar}_G(B_GU(n))(G/G)$, according to the Greenlees-May decomposition, has $n^2+2n$ many algebra generators over the homology of a point, which is just under double the minimal amount $\frac{n^2+2n}2+1$ of generators that we can obtain (see the remarks after Proposition \ref{C2Chern2}). Part of the goal of this paper is to systematically obtain such minimal explicit descriptions; said another way, we are producing only the essential characteristic classes upon which all the others are built.
	
	Our method rests on equivariant generalizations of the following nonequivariant arguments: By a classical Theorem of Borel (\cite{BCM}), if $L$ is a connected compact Lie group, $T\subseteq L$ a maximal torus and $W_LT=N_LT/T$ is the Weyl group then, at least rationally,
	\begin{equation}
		H^*(BL)=H^*(BT)^{W_LT}
	\end{equation}
	Through this result, the characteristic classes in $H^*(BL)$ can be computed from $H^*(BS^1)$, as long as the Weyl group action is understood. For example, if we take $L=U(n)$ then $T=(S^1)^n$ and $W_LT=\Sigma_n$ acts on $H^*(BT;\Q)=\Q[a_1,...,a_n]$ by permuting the generators $a_i$. The fixed points under this permutation action are minimally generated by the elementary symmetric polynomials on the $a_i$, which are by definition the Chern classes $c_i$. In this way, $H^*(BU(n);\Q)=\Q[c_1,...,c_n]$.
	
	The same method can be performed equivariantly for $G=C_2$ and coefficients in $A_{\Q}$. There is an extra degree of complexity owing to the fact that $H^{\bigstar}_G(B_GS^1;A_{\Q})$ is not polynomial on one generator over $H^{\bigstar}_G(*;A_{\Q})$, but rather on two generators, one of which is idempotent (Proposition \ref{C2Chern1Class}). As such, in the $L=U(n)$ example, the elementary symmetric polynomials $c_i$ must be replaced by a family of more complicated polynomials $\genu,\gene_i,\genc_{s,j}$ (Proposition \ref{C2Chern3}). Moreover, while this family of generators is minimal, it is not algebraically independent i.e.  there are relations within this family. It is true however that $H^{\bigstar}_G(B_GU(n);A_{\Q})$ is a finite module over $H^{\bigstar}_G(*;A_{\Q})[\gene_1,...,\gene_n]$ where the $\gene_i$ are $C_2$-equivariant refinements of the classical Chern classes.
	
	We use this method to obtain explicit minimal descriptions of $H^{\bigstar}_{C_2}(B_{C_2}L;A_{\Q})$ where $L=U(n), SO(n), Sp(n)$. We also examine the cases of $L=O(n), SU(n)$ and of the non-compact Lie groups $L=U, SO, Sp, O, SU$. The resulting equivariant Chern, Pontryagin and symplectic classes are compared using the complexification, quaternionization and forgetful maps between the aforementioned Lie groups. Finally, we compute the effect of these characteristic classes on the direct sum of bundles and on the tensor product of line bundles.
	
	As for the organization of this paper, sections \ref{Conv} and \ref{C2RationalStems} set up the notation used throughout and contain the computation of the $C_2$ rational stable stems.
	
	Section \ref{Summary} contains a summary of all our results on $C_2$ characteristic classes. The proofs are then found in sections \ref{C2ChernSection}-\ref{SU} for the interested reader.
	
	Finally, appendix \ref{appen} contains the results on symmetric polynomials with relations that are critical for our presentation of $H^{\bigstar}_{C_2}(B_{C_2}U(n);A_{\Q})$. In particular, it contains an algorithm for writing every "symmetric polynomial" in terms of the "elementary symmetric polynomials" $\genc_{s,j}$; this also leads to an algorithm for explicitly obtaining the relations between the $\genc_{s,j}$. We have implemented these algorithms in a computer program available \href{https://github.com/NickG-Math/Symmetric_Polynomials}{here} (executable files are available \href{https://github.com/NickG-Math/Symmetric_Polynomials/releases}{here} for a quick demonstration). The appendix is completely self contained and independent of the rest of the paper. 
	
	\subsection{Acknowledgment}  We would like to thank Peter May for reading several earlier drafts of this paper. Through his numerous editing suggestions, readability was vastly improved.

	\section{Conventions and Notations}\label{Conv}
	
Throughout this paper, the ambient group is $G=C_2$ and all our $G$-Mackey functors are modules over the rational Burnside Green functor $A_{\Q}$:
\begin{equation}
	A_{\Q}=\begin{tikzcd}
		\frac{\Q[x]}{x^2=2x}\ar[d, "x\mapsto 2" left, bend right]\\
		\Q\ar[u, "1\mapsto x" right,bend right]
	\end{tikzcd}=\begin{tikzcd}
		\Q x\ar[d, "x\mapsto 2" left, bend right]\\
		\Q\ar[u, "1\mapsto x" right,bend right]
	\end{tikzcd}
	\oplus 
	\begin{tikzcd}
		\Q y\ar[d, bend right]\\
		0\ar[u, bend right]
	\end{tikzcd}
\end{equation}
where $x=\Tr(1)$ and $y=1-x/2$.\medbreak

The \emph{unreduced} cohomology of a $G$-space $X$ in $A_{\Q}$ coefficients is the $G$-Green functor defined on orbits as
\begin{equation}
	H^{\bigstar}_G(X)(G/H)=	[X_+, \Sigma^{\bigstar}HA_{\Q}]^H
\end{equation}
where $HA_{\Q}$ is the Eilenberg-MacLane spectrum associated to $A_{\Q}$ and the index $\bigstar$ is an element of the real representation ring $RO(G)=RO(C_2)$. This ring is spanned by the trivial representation $1$ and the sign representation $\sigma$ so $\bigstar=n+m\sigma$ for $n,m\in \Z$. Moreover, $H^{\bigstar}_G(X)$ is a Green functor algebra over the cohomology of a point $H^{\bigstar}_G(*)=H^{\bigstar}_G$.\medbreak
		
The same conventions apply to homology $H^G_{\bigstar}(X)$ (with the exception of the ring structure, which exists only when $X$ is an equivariant $H$-space).\medbreak
	
The advantage of using $A_{\Q}$ coefficients is twofold: \cite{GM95} prove that
	\begin{itemize}
	\item All rational Mackey functors (i.e. $A_{\Q}$ modules) are projective and injective, so we have the Kunneth formula:
		\begin{equation}
H_{\bigstar}^G(X\times Y)=H_{\bigstar}^G(X)\boxtimes_{H_{\bigstar}^G}H_{\bigstar}^G(Y)
\end{equation}
and duality formula:
		\begin{equation}
H_G^{\bigstar}(X)=\Hom_{H_{\bigstar}^G}(H_{\bigstar}^G(X),H_{\bigstar}^G)
\end{equation}
	\item We have the isomorphism of graded Green functors:
	\begin{equation}
		H_G^*(X)=\begin{tikzcd}
			H^*(X)^G\ar[d, bend right]\\
			H^*(X)\ar[u,bend right]
		\end{tikzcd}\oplus
		\begin{tikzcd}
			H^*(X^G)\ar[d, bend right]\\
			0\ar[u, bend right]
		\end{tikzcd}
	\end{equation}
\end{itemize}
The second bullet allows us to reduce equivariant computations to nonequivariant ones, as long as we use integer grading $*\in \Z$. Using the first bullet, integer graded cohomology together with the homology of a point recover the $RO(G)$-graded cohomology:
\begin{equation}
	H_G^{\bigstar}(X)=H_G^*(X)\boxtimes_{A_{\Q}}H_G^{\bigstar}
\end{equation}
As such, once $H_G^{\bigstar}$ is computed, we need only worry about integer grading.
	
		\section{\texorpdfstring{The $C_2$ rational stable stems}{The C2 rational stable stems}}\label{C2RationalStems}
The Green functor $H_{\bigstar}^G=H^{-\bigstar}_G$ agrees with the $G$-equivariant rational stable stems:
\begin{equation}
\pi_{\bigstar}^G(S)\otimes \Q=\pi_{\bigstar}^G(HA_{\Q})=H_{\bigstar}^G
\end{equation}
The generating classes for $H_{\bigstar}^G$ are the Euler and orientation classes. The Euler class $a_{\sigma}$ is the inclusion of north-south poles $S^0\hookrightarrow S^{\sigma}$ and its image in $H^{\bigstar}_G$ under the Hurewicz map generates a Mackey functor that we denote by $M_1$:
\begin{equation}
M_1\{a_{\sigma}\}=\begin{tikzcd}
\Q a_{\sigma}\ar[d, bend right]\\
0\ar[u, bend right]
\end{tikzcd}
\end{equation}
The orientation class $u_{\sigma}$ is the generator of the reduced nonequivariant homology group $\tilde H_1(S^{\sigma};\Z)=\Z$ (determined uniquely once we fix an orientation for $S^{\sigma}$) and generates a Mackey functor that we denote by $M_0^-$:
\begin{equation}
M_0^-\{u_{\sigma}\}=\begin{tikzcd}
0\ar[d, bend right]\\
\Q u_{\sigma}\ar[u, bend right]\ar[loop right, "C_2"]
\end{tikzcd}
\end{equation}
The Weyl group action by the generator $g\in C_2$ is $gu_{\sigma}=-u_\sigma$.

The square of $u_{\sigma}$, $u_{\sigma}^2$, is the restriction of the orientation class $u_{2\sigma}$ generating a Mackey functor that we denote by $M_0$:
\begin{equation}
M_0\{u_{2\sigma}\}=\begin{tikzcd}
\Q u_{2\sigma}\ar[d, "1" left, bend right]\\
\Q u^2_{\sigma}\ar[u, "2" right, bend right]
\end{tikzcd}
\end{equation}
This follows from the fact $M_0^-\boxtimes_{A_{\Q}} M_0^-=M_0$ and by the Kunneth formula for $S^{2\sigma}=S^{\sigma}\wedge S^{\sigma}$. Note that $a_{\sigma}u_{2\sigma}=0$ since $M_1\boxtimes_{A_{\Q}}M_0=0$.

Using the duality formula:
			\begin{equation}
\tilde H_{-*}^G(S^{-\sigma})=\tilde H_G^{*}(S^{\sigma})=\Hom_{A_{\Q}}(\tilde H_*^G(S^{\sigma}),A_{\Q})
	\end{equation}
%	From this duality, 
we see that there is a class generating $M_1$ which when multiplied with $a_{\sigma}$ returns $y\in A_{\Q}$; we denote this class by $y/a_{\sigma}$. We similarly have classes $u_{\sigma}^{-1}$ and $x/u_{2\sigma}$ spanning $M_0^-$ and $M_0$ respectively. We have proven the following Proposition:

	\begin{prop}The $C_2$ equivariant rational stable stems are:
		\begin{equation}
	H^G_{k+n\sigma}=\begin{cases}
		M_0&\textup{if }k=n\text{ : even and}\neq 0\\
		M_0^{-}&\textup{if }k=n\text{ : odd} \\
		M_1&\textup{if }k=0\text{ , }n\neq 0\\
		A_{\Q}&\textup{if }k=n=0\\
		0&\textup{otherwise }
		\end{cases}
		\end{equation}
	and:
	\begin{itemize}
		\item $u_{2\sigma}^{j}, x/u_{2\sigma}^j$ generate a copy $M_0$ for each $j=1,2,...$.
		\item $u_{\sigma}^{2j+1}$ generate a copy $M_0^-$ for each $j\in \Z$.
		\item $a_{\sigma}^{j}, y/a_{\sigma}^j$ generate a copy of $M_1$ for each $j=1,2,...$.
		\item $1$ generates $A_{\Q}$.
	\end{itemize}
	\end{prop}	
	To spell things out, as a ring, the $C_2/C_2$ level of $H^G_{\bigstar}$ is $\Q[x, u_{2\sigma}, x/u_{2\sigma}, a_{\sigma}, y/a_{\sigma}]$ modulo the relations:
	\begin{gather}
	x^2=2x\\
	xu_{2\sigma}=2u_{2\sigma}\\
	ya_{\sigma}=a_{\sigma}\\
	u_{2\sigma}(x/u_{2\sigma}^i)=x/u_{2\sigma}^{i-1}\\
	a_{\sigma}(y/a_{\sigma}^i)=y/a_{\sigma}^{i-1}\\
	a_{\sigma}u_{2\sigma}=0\\
	a_{\sigma}(x/u_{2\sigma}^i)=0\\
	u_{2\sigma}(y_1/a_{\sigma}^i)=0\\
	(x/u_{2\sigma}^i)(y/a_{\sigma}^j)=0
	\end{gather}

	\section{\texorpdfstring{Summary of the $C_2$ characteristic classes}{Summary of the C2 characteristic classes}}\label{Summary}
	
	We summarize our results on $C_2$ characteristic classes in $A_{\Q}$ coefficients that we shall prove in sections \ref{C2ChernSection}-\ref{SU}. 
	
	Slightly abusing the notation, we shall use $H^*_G(X)$ to simultaneously denote both the $G$-Green functor and its top level $H^*_G(X)(G/G)$.
	We can do that because knowledge of the top and bottom levels and of the restriction map can be used to recover the Mackey functor, as long as the restriction is surjective (the transfer is computed from $\Tr(a)=xb$ where $\Res(b)=a$). In all cases we encounter, restriction is indeed surjective so it suffices to describe the top level and how generators restrict to the bottom (nonequivariant) level.
	
	\subsection{Chern classes}

	We start with the results on Chern classes.
	
	We view $H^*_G(B_GU(n))$ as an augmented algebra over $H^*(BU(n))$ with the augmentation being restriction.
	
	\begin{prop}\label{C2Chern1}The augmentation
		\begin{equation}
		\Res:H^*_G(B_GU(n))\to H^*(BU(n))
		\end{equation}
		is a split surjection, so the nonequivariant Chern classes have $C_2$ equivariant refinements.
	\end{prop}
We fix a section of the augmentation, i.e. equivariant refinements $\gene_1,...,\gene_n$ of the Chern classes, according to Proposition \ref{C2ChernIso}.
	
	\begin{prop}\label{C2Chern2}There exist elements $\genu\in H^0_G(B_GU(n))$ and $\genc_{s,j}\in H^{2s}_G(B_GU(n))$ for $1\le s<n$ and $1\le j\le n-s$,  generating $H_G^*(B_GU(n))$ as an augmented algebra over $H^*(BU(n))\otimes A_{\Q}$:
		\begin{equation}
		H_G^*(B_GU(n))=\frac{(H^*(BU(n))\otimes A_{\Q})[\genu,\genc_{s,j}]}{\Res(\genu), \Res(\genc_{s,j}), S}
		\end{equation}
		where the finite set of relations $S\subseteq \Q[\genu,\gene_i,\genc_{s,j}]$ is described in Proposition \ref{AlgebraCorollaryQ}. \\
		The $\gene_i$ are algebraically independent and for each degree $*$, $H_G^*(B_GU(n))$ is a finitely generated module over $\Q[\gene_1,...,\gene_n]$.\\
		The generating family $\{\genu, \gene_i,\genc_{s,j}\}$ has cardinality $\frac{n^2+2n}2+1$ and is a minimal generating set of $H_G^*(B_GU(n))$ as an $A_{\Q}$ algebra, in the sense that any other generating set has at least $\frac{n^2+2n}2+1$ many elements.	
	\end{prop}
	
Substituting $H^*(BU(n))=\Q[c_1,...,c_n]$ in the formula for $H^*_G(B_GU(n))$ gives:
	\begin{prop}\label{C2Chern3}As an algebra over $A_{\Q}$,
		\begin{equation}
		H_G^*(B_GU(n))=\frac{A_{\Q}[\genu, \gene_i,\genc_{s,j}]}{x\genu,x\genc_{s,j},S}
		\end{equation}
	\end{prop}
	
	Two observations:
	\begin{itemize}
		\item The relations $x\genu=0, x\genc_{s,j}=0$ are equivalent to $\genu,\genc_{s,j}$ having trivial restrictions (i.e. augmentations) respectively. This completes the description of the Mackey functor structure of $H^*_G(B_GU(n))$.
		\item The $\frac{n^2+2n}2+1$ many generators of the generating set $\{\genu, \gene_i,\genc_{s,j}\}$ are just over half of the $n^2+2n$ many generators given by the idempotent decomposition (\cite{GM95}) of the Mackey functor $H_G^*(B_GU(n))$.\medbreak
	\end{itemize}
	
	For $n=1$ the computation takes a simpler form:
	\begin{equation}
	H_G^*(B_GU(1))=A_{\Q}[\genu,\gene_1]/(\genu^2=\genu, x\genu)
	\end{equation}
	To simplify the notation in the next Proposition, we set $\genw=\gene_1\in H^2_G(B_GU(1))$.
	
	\begin{prop}\label{C2ChernIso}	The maximal torus inclusion $U(1)^n\hookrightarrow U(n)$ induces an isomorphism
		\begin{gather}
		H^*_G(B_GU(n))=	(H^*_G(B_GU(1))^{\otimes n})^{\Sigma_n}
		\end{gather}	
		Explicitly:
		\begin{equation}
		A_{\Q}[\genu,\gene_i,\genc_{s,j}]/(x\genu, x\genc_{s,j}, S)=(A_{\Q}[\genu_i,\genw_i]/(x\genu_i))^{\Sigma_n}
		\end{equation}
		under the identifications:
		\begin{gather}
		\genu=\sigma_1(\genu_1,...,\genu_n)=\sum_{1\le m\le n}\genu_m\\
		\gene_i=\sigma_i(\genw_1,...,\genw_n)=\sum_{m_*\in K_i}\genw_{m_1}\cdots \genw_{m_i}\\
		\genc_{s,j}=\sum_{(m_*,l_*)\in K_{s,j}}\genw_{m_1}\cdots \genw_{m_s}\genu_{l_1}\cdots \genu_{l_j}
		\end{gather}
		where $K_i$ consists of all partitions $1\le m_1<\cdots<m_i\le n$ and $K_{s,j}\subseteq K_s\times K_j$ consists of all pairs of disjoint partitions. The polynomial $\sigma_i$ is the $i$-th elementary symmetric polynomial.
	\end{prop}
The family of generators $\genu,\gene_i,\genc_{s,j}$ is determined upon choosing $\genu,\genw=\gene_1$ in $H^0_G(B_GU(1)), H^2_G(B_GU(1))$ respectively, with:
	\begin{equation}
H_G^*(B_GU(1))=A_{\Q}[\genu,\genw]/(\genu^2=\genu, x\genu)
\end{equation}
The choice of $\genw$ is unique under the additional requirement that its restriction is the nonequivariant Chern class $\gene_1$ (in this way, the equivariant $\gene_i$ are all canonically determined). There are two equally good candidates for $\genu$ however: $\genu$ and $y-\genu$. They can only be distinguished upon fixing a model for $B_GU(1)$, as we do in subsection \ref{n=1Comp}. As such, there is no canonical choice of $\genu\in H^0_G(B_GU(1))$.

	\begin{prop}\label{AddTriv}The map $B_GU(n)\to B_GU(n+1)$ given by direct sum with a trivial complex representation induces on cohomology:
	\begin{gather}
	\genu\mapsto y+\genu\\
	\gene_i\mapsto \gene_i\\
	\genc_{s,j}\mapsto \genc_{s,j}+\genc_{s,j-1}\end{gather}
	using the convention $\genc_{s,0}=yc_s$. \\
		The map $B_GU(n)\to B_GU(n+1)$ given by direct sum with a $\sigma$ representation induces on cohomology:
	\begin{gather}
	\genu\mapsto \genu\\
	\gene_i\mapsto \gene_i\\
	\genc_{s,j}\mapsto \genc_{s,j}
	\end{gather}
	For both maps we use the conventions that $\gene_{n+1}=0$ and $\genc_{s,n+1-s}=0$ in every RHS.
\end{prop}	
\begin{prop}\label{AddBun}
	The direct sum of bundles map $B_GU(n)\times B_GU(m)\to B_GU(n+m)$ induces on cohomology: 
	\begin{gather}
	\genu\mapsto \genu\otimes 1+1\otimes \genu\\
	\gene_i\mapsto \sum_{j+k=i}\gene_j\otimes \gene_k\\
	\genc_{s,j}\mapsto \sum_{s'+s''=s\atop j'+j''=j}\genc_{s',j'}\otimes \genc_{s'',j''}
	\end{gather} 
	using the conventions $\gene_0=1,\genc_{s,0}=yc_s,\genc_{0,j}=(j!)^{-1}\genu(\genu-1)\cdots (\genu-j+1)$ in every RHS.
\end{prop}

\begin{prop}\label{TensorBun}	
	The tensor product of line bundles map $B_GU(1)\times B_GU(1)\to B_GU(1)$ induces on cohomology:
	\begin{gather}
	\genu\mapsto y-\genu\otimes 1-1\otimes \genu+2\genu\otimes \genu\\
	\gene_1\mapsto \gene_1\otimes 1+1\otimes \gene_1\end{gather}
\end{prop}

\subsection{Symplectic classes}The theory of $C_2$ symplectic characteristic classes is entirely analogous to Chern classes, by replacing $B_GU(n)$ with $B_GSp(n)$ and the generators $\gene_i,\genc_{s,j}$ with generators $k_i,\kappa_{s,j}$ of double degree. Propositions \ref{C2Chern1}-\ref{C2Chern3} become:

\begin{prop}There exist classes $\genu,k_i,\kappa_{s,j}\in H_G^*(B_GSp(n))$ of degrees $0,4i,4s$ respectively, where $1\le i,s\le n$ and $1\le j\le n-s$, such that
	\begin{equation}
	H_G^*(B_GSp(n))=\frac{A_{\Q}[\genu,k_i,\kappa_{s,j}]}{x\genu, x\kappa_{s,j}, S}
	\end{equation}
	where the relation set $S$ is the same as that for $H_G^*(B_GU(n))$ with $\gene_i,\genc_{s,i}$ replaced by $k_i,\kappa_{s,i}$.
		
	The generators $k_i$ restrict to the nonequivariant symplectic classes $k_i$, so the restriction map $H^*_G(B_GSp(n))\to H^*(BSp(n))$ is a split surjection.
	
	The maximal torus inclusion $U(1)^n\hookrightarrow Sp(n)$ induces an isomorphism
	\begin{gather}
	H^*_G(B_GSp(n))=	(H^*_G(B_GU(1))^{\otimes n})^{C_2\wr \Sigma_n}\end{gather}
Explicitly:
\begin{gather}
	A_{\Q}[\genu,k_i,\kappa_{s,j}]/(x\genu, x\kappa_{s,j}, S)=(A_{\Q}[\genu_i,\genw_i]/(x\genu_i))^{C_2\wr\Sigma_n}
	\end{gather}
	under the identifications: 
	\begin{gather}
	\genu=\sum_{1\le m\le n}\genu_m\\
	k_i=\sum_{m_*\in K_i}\genw_{m_1}^2\cdots \genw_{m_i}^2\\
	\kappa_{s,j}=\sum_{(m_*,l_*)\in K_{s,j}}\genw_{m_1}^2\cdots \genw_{m_s}^2\genu_{l_1}\cdots \genu_{l_j}
	\end{gather}
	where $K_i$ and $K_{s,j}$ are as in Proposition \ref{C2ChernIso}. 
\end{prop}

Propositions \ref{AddTriv}-\ref{TensorBun} have analogous statements in the symplectic case, replacing $B_GU(n)$ by $B_GSp(n)$ and $\gene_i,\genc_{s,j}$ with $k_i,\kappa_{s,j}$ respectively; we shall not repeat them here.

\begin{prop}The forgetful map $B_GSp(n)\to B_GU(2n)$ induces on cohomology:
	\begin{gather}
	\genu\mapsto \genu\\
	\gene_{2i+1}, \genc_{2s+1,j}\mapsto 0\\
	\gene_{2i}\mapsto (-1)^ik_i\\
	\genc_{2s,j}\mapsto (-1)^s\kappa_{s,j}
	\end{gather}
	The quaternionization map $B_GU(n)\to B_GSp(n)$ induces:
	\begin{gather}
	\genu\mapsto \genu\\
	k_i\mapsto \sum_{a+b=2i}(-1)^{a+i}\gene_a\gene_b
	\end{gather}
	The effect of quaternionization on the $\kappa_{s,j}$ is explained in Proposition \ref{QuaterExplained}.
\end{prop}

\subsection{Pontryagin and Euler classes} The results are analogous to the symplectic classes, but we need to distinguish between $B_GSO(2n)$ and $B_GSO(2n+1)$. The following Proposition contains the shared aspects of both cases:

	\begin{prop}
	The restriction map $H^*_G(B_GSO(n))\to H^*(BSO(n))$ is a split surjection.	The maximal torus inclusion $T\hookrightarrow SO(n)$ induces an isomorphism
	\begin{gather}
		H^*_G(B_GSO(n))=(H^*_G(B_GT))^W
	\end{gather}
	where $W$ is the corresponding Weyl group.
\end{prop}

This gives us $C_2$ equivariant refinements $p_i,\chi$ of the Pontryagin and Euler classes respectively. Recall that for $BSO(2n)$ the characteristic classes are $p_1,...$,$p_{n-1}$, $\chi$ (and $p_n=\chi^2$) while for $BSO(2n+1)$ they are $p_1,...,p_n$.

\begin{prop}There exist classes $\genu,\pi_{s,j}$ of degrees $0,4s$ respectively in $H^*_G(B_GSO(2n))$, where $1\le s< n$ and $1\le j\le n-s$ such that
	\begin{equation}
	H_G^*(B_GSO(2n))=\frac{A_{\Q}[\genu,p_i,\pi_{s,j},\chi]}{x\genu, x\pi_{s,j}, S}
	\end{equation}
	where the relation set $S$ is the same as that for $H_G^*(B_GU(n))$ with $\gene_i,\genc_{s,i}$ replaced by $p_i,\pi_{s,i}$ and using that $p_n=\chi^2$.
	
	Under the maximal torus isomorphism:
		\begin{gather}
\genu=\sum_{1\le m\le n}\genu_m\\
p_i=\sum_{m_*\in K_i}\genw_{m_1}^2\cdots \genw_{m_i}^2\\
\pi_{s,j}=\sum_{(m_*,l_*)\in K_{s,j}}\genw_{m_1}^2\cdots \genw_{m_s}^2\genu_{l_1}\cdots \genu_{l_j}\\
\chi=\genw_1\cdots \genw_n
\end{gather}
		where $K_i$ and $K_{s,j}$ are as in Proposition \ref{C2ChernIso}. 
\end{prop}
	
	\begin{prop}
	The map $B_GSO(2n)\to B_GSO(2n+1)$ induces an injection in cohomology and:
	\begin{equation}
	H_G^*(B_GSO(2n+1))=\frac{A_{\Q}[\genu,p_i,\pi_{s,j}]}{x\genu, x\pi_{s,j}, S}
	\end{equation}
where $i=1,...,n$.
\end{prop}

Propositions \ref{AddTriv}-\ref{TensorBun} have analogous statements in this context. The action on the Euler class $\chi$ is the same as in the nonequivariant case; for example, under $B_GSO(n)\times B_GSO(m)\to B_GSO(n+m)$ we get:
\begin{equation}
\chi\mapsto \chi\otimes\chi
\end{equation}

\begin{prop}The complexification map $B_GSO(2n)\to B_GU(2n)$ induces on cohomology:
	\begin{gather}
	\genu\mapsto \genu\\
	\gene_{2i+1}, \genc_{2s+1,j}\mapsto 0\\
	\gene_{2i}\mapsto (-1)^ip_i\\
	\genc_{2s,j}\mapsto (-1)^s\pi_{s,j}
	\end{gather}
	The forgetful map $B_GU(n)\to B_GSO(2n)$ induces on cohomology:
	\begin{gather}
	\genu\mapsto \genu\\
	p_i\mapsto \sum_{a+b=2i}(-1)^{a+i}\gene_a\gene_b\\
	\chi\mapsto \gene_n
	\end{gather}
and the action on $\pi_{s,j}$ is explained in Proposition \ref{ForgetExplained}.
\end{prop}

\subsection{Stable characteristic classes}\label{StableSummary}

In the $C_2$-equivariant case, there are different notions of stability for complex bundles, represented by the following spaces:
\begin{itemize}
	\item $B_G^+U=\colimit (B_GU(1)\xrightarrow{\oplus 1}B_GU(2)\xrightarrow{\oplus 1}\cdots)$. This is the usual equivariant classifying space $B_GU=E_GU/U$ and is a $G$-equivariant $H$-space using the direct sum of bundles maps $B_GU(n)\times B_GU(m)\to B_GU(n+m)$.
	\item $B_G^-U=\colimit (B_GU(1)\xrightarrow{\oplus \sigma }B_GU(2)\xrightarrow{\oplus \sigma}\cdots)$. This is equivalent to $B_G^+U$.
	\item $B_G^{\pm}U=\colimit (B_G^+U\xrightarrow{\oplus \sigma}B_G^+U\xrightarrow{\oplus \sigma}\cdots)=\colimit (B_G^-U\xrightarrow{\oplus 1}B_G^-U\xrightarrow{\oplus 1}\cdots)$. This becomes a $G$-equivariant $H$-space using the direct sum of bundles, and is the group completion of $B_G^+U$ (and $B_G^-U$). Moreover, $B_G^{\pm}U\times \Z$ represents equivariant $K$-theory.
\end{itemize}
Computing $H^*_G(B_G^-U)$ in terms of the generators $\genu,\gene_i,\genc_{s,j}$ is more complicated compared to the nonequivariant case because for fixed degree $*$, the $\Q$-dimension of $H^*_G(B_GU(n))$ does not stabilize as $n\to +\infty$ and as a result, $H^*_G(B_G^-U)$ is infinite dimensional (dimension is $2^{\aleph_0}$). In degree $*=0$, $H^0_G(B_G^-U)$ is linearly spanned over $A_{\Q}$ by series of the form 
\begin{equation}
a_{-1}+	\sum_{i\ge 0}a_i\genu(\genu-1)\cdots (\genu-i)
\end{equation}
Generally, the graded algebra $H^*_G(B_G^-U)$ is generated over $H^0_G(B_G^-U)[\gene_1,\gene_2,...]$ by series of the form
\begin{equation}
\sum_{j=1}^{\infty}a_j\genc_{s,j}\in H^{2s}_G(B_G^-U)
\end{equation}
for $a_j\in \Q$ and $s=1,2,...$. See section \ref{C2ChernStable} for more details.

For the ring $H^*_G(B_G^{\pm}U)$ we also have to compute the effect of the $\oplus 1$ map on the series in $H^*_G(B_G^-U)$. If we restrict our attention to finite series, we are in essence dealing with characteristic classes that are stable under addition of both the $\oplus 1 $ and $\oplus \sigma$ representations. Since the $\oplus 1$ map takes the form $\genc_{s,j}\mapsto \genc_{s,j}+\genc_{s,j-1}$ (and $\genc_{s,0}=yc_s$, $\genc_{0,1}=\genu$) we can immediately see that for $i\ge 1$, the elements
\begin{equation}
	\gene_i\text{ , }\genc_i:=\gene_i\genu-\genc_{i,1}
\end{equation}
are stable under both $\oplus 1$ and $\oplus \sigma$. We conjecture that all classes with this property are polynomially generated by $\gene_i,\genc_i$; this is equivalent to the elements $\genc_1,\genc_2,...$ being algebraically independent over $\Q[\gene_1,\gene_2,...]$.

In any case, the elements $\gene_i,\genc_i$ span sub-Hopf-algebras of $H^*_G(B_G^-U)$ and $H^*_G(B_G^{\pm}U)$ with
	\begin{gather}
	\genc_s\mapsto \sum_{i+j=s}(\gene_i\otimes \genc_j+\genc_i\otimes \gene_j)
\end{gather} 
using the conventions $\gene_0=1$ and $\genc_0=0$.\medbreak
\iffalse
The equivariant cohomology of $B_G^+U$ is the limit of $H^*_G(B_GU(n))$ i.e. an element of $H^*_G(B_G^+U)$ is a sequence $s_n\in H^*_G(B_GU(n))$ satisfying a compatibility condition (namely $s_n\mapsto s_{n-1}$ under $H^*_G(B_GU(n))\xrightarrow{\oplus 1} H^*_G(B_GU(n-1))$); similarly an element of $H^*_G(B_G^{\pm}U)$ is a compatible double indexed sequence $s_{n,m}\in H^*_G(B_GU(n))$. Describing these sequences in terms of the generators $\genu,\gene_i,\genc_{s,j}$ is rather complicated (see section \ref{C2ChernStable}). The constant sequences are simple enough however:
	\begin{prop}\label{SuperStableChern}The elements $\gene_i$ and $\genc_i:=\gene_i\genu-\genc_{i,1}$ for $i\ge 1$, are invariant under both $H^*_G(B_GU(n))\xrightarrow{\oplus 1} H^*_G(BU(n-1))$ and $H^*_G(B_GU(n))\xrightarrow{\oplus \sigma} H^*_G(BU(n-1))$. They span the sub-Hopf-algebra of $H^*_G(B_G^{\pm}U)$ consisting of constant sequences, given by:
	\begin{equation}
		\frac{A_{\Q}[\gene_i,\genc_i]}{x\genc_i} 
	\end{equation}
with coalgebra structure:
	\begin{gather}
	\gene_s\mapsto \sum_{i+j=s}\gene_i\otimes \gene_j\\
	\genc_s\mapsto \sum_{i+j=s}(\gene_i\otimes \genc_j+\genc_i\otimes \gene_j)
\end{gather} 
using the conventions $\gene_0=1$ and $\genc_0=0$.
	\end{prop}
\fi
The spaces $B_G^+U, B_G^-U, B_G^{\pm}U$ are equivariant $H$-spaces, hence their equivariant homology is a Green functor dual to their equivariant cohomology. This homology can be expressed in terms of the classes $a_i,b_i,d\in H_*^G(B_GU(1))$ dual to $\genu\gene_1^i,\gene_1^i, x/2+\genu\in H^*_G(B_GU(1))$ respectively, where $i\ge 1$. Note that the $\genc_1^i\in H^*_G(B_GU(2))$ map to $\genu\gene_1^i$ under $H^*_G(B_GU(2))\xrightarrow{\oplus 1}H^*_G(B_GU(1))$ so the $a_i$ can be thought of as duals to the $\genc_1^i$.
	\begin{prop}\label{StableChernHomology}We have:
	\begin{gather}
		H_*^G(B_G^-U)=\frac{A_{\Q}[d,a_i,b_i]}{xa_i, xd=x}\\
		H_*^G(B_G^{\pm}U)=\frac{A_{\Q}[d^{\pm},a_i,b_i]}{xa_i, xd=x}
	\end{gather}
and for the coalgebra structure:
	\begin{gather}
	d\mapsto d\otimes d\\
	a_i\mapsto \sum_{j+k=i}a_j\otimes a_k\\
	b_i\mapsto \sum_{j+k=i}(b_j\otimes b_k-b_j\otimes a_k-b_k\otimes a_j+2a_j\otimes a_k)
\end{gather}
using the conventions $a_0=d-x/2$ and $b_0=1$.
\end{prop}
The case of stable symplectic classes is entirely analogous: We can distinguish between $B_G^+Sp, B_G^-Sp$ and $ B_G^{\pm}Sp$ and we have classes $k_i, \kappa_i=k_i\genu-\kappa_{i,1}$ that are stable under both $\oplus 1, \oplus \sigma$ maps. Moreover,
\begin{prop}\label{UvsSp}
	The forgetful map $Sp\to U$ induces
\begin{gather}
\gene_{2s+1}, \genc_{2s+1}\mapsto 0\\
\gene_{2s}\mapsto (-1)^sk_{s}\\
\genc_{2s}\mapsto (-1)^s\kappa_{s}
\end{gather}
while quaternionization $U\to Sp$ induces
\begin{gather}
k_i\mapsto \sum_{a+b=2i}(-1)^{a+i}\gene_a\gene_b\\
\kappa_i\mapsto \sum_{a+b=2i}(-1)^{a+i}\gene_a\genc_b
\end{gather}
\end{prop}
The dual homology result can be expressed in terms of the classes $a_i^{sp},b_i^{sp},d\in H_*^G(B_GSp(1))$ dual to $\genu k_1^i,k_1^i, x/2+\genu\in H^*_G(B_GSp(1))$ respectively, for $i\ge 1$ (the $a_i^{sp}$ are dual to $\kappa_1^i$). The analogue of Proposition \ref{StableChernHomology} holds, and:
\begin{prop}\label{UvsSpHomology}
	The forgetful map $Sp\to U$ induces
	\begin{gather}
	d\mapsto d\\
	a_i^{sp}\mapsto \sum_{2i=j+k}(-1)^ka_ja_k\\
	b_i^{sp}\mapsto \sum_{2i=j+k}(-1)^k(b_jb_k-a_jb_k-a_kb_j+2a_ja_k)
	\end{gather}
	while quaternionization $U\to Sp$ induces
	\begin{gather}
	d\mapsto d\\
	a_{2i+1}\mapsto 0\text{ , }a_{2i}\mapsto a_i^{sp}\\
	b_{2i+1}\mapsto 0\text{ , }b_{2i}\mapsto b_i^{sp}
	\end{gather}
\end{prop}

The case of stable Pontryagin classes is entirely analogous, replacing $Sp$ by $SO$ (the forgetful map $Sp\to U$ is replaced by complexification $SO\to U$ and the quaternionization map $U\to Sp$ is replaced by the forgetful map $U\to SO$). In brief, setting $\pi_i=p_i\genu-\pi_{i,1}$ gives the analogue of \ref{UvsSp}. Moreover, we have classes $a_i^{so},b_i^{so},d\in H_*^G(B_GSO(2))$ dual to $\genu p_1^i,p_1^i, x/2+\genu\in H^*_G(B_GSO(2))$ respectively, for $i\ge 1$, and the analogues of Propositions \ref{StableChernHomology} and \ref{UvsSpHomology} also hold.

\subsection{Orthogonal groups} Unlike their nonequivariant counterparts, the $C_2$ classifying spaces of the orthogonal groups $O(n)$ don't generally satisfy the maximal torus isomorphism, i.e. $H^*_G(B_GO(n))\to H_G^*(B_GT)^W$ is not generally an isomorphism, where $T$ is the maximal torus in $O(n)$ and $W$ the Weyl group. Moreover, $H^*_G(B_GO(2n))$ is not isomorphic to $H^*_G(B_GO(2n+1))$, but rather, the inclusion-induced map $$H^*_G(B_GO(2n+1))\to H^*_G(B_GO(2n))$$ is always a surjection with nontrivial kernel. The spaces $B_GO(2n+1)$ can be put into our framework using the splitting $O(2n+1)=SO(2n+1)\times O(1)$:
\begin{prop}There is a generator $\gend\in H^0_G(B_GO(1))$ such that
	\begin{equation}
		H^*_G(B_GO(2n+1))=\frac{A_{\Q}[\genu,\gend,p_i,\pi_{s,j}]}{x\genu, x\gend, x\pi_{s,j}, S}
	\end{equation}
\end{prop}
The $H^*_G(B_GO(2n))$ can then be understood as quotients of $H_G^*(B_GO(2n+1))$ (see section \ref{Orthogonal}). The stable case similarly reduces to $B_GSO$ by use of the fact that $B_GO=B_GSO\times B_GO(1)$.

\subsection{Special unitary groups} For $SU(n)$ we have the maximal torus isomorphism equivariantly:
	\begin{prop}	The maximal torus inclusion $U(1)^{n-1}\to SU(n)$ induces an isomorphism 
	\begin{equation}
	H^*_G(B_GSU(n))\to H^*_G(B_GU(1)^{n-1})^{\Sigma_n}
	\end{equation}
\end{prop}
We prove that for any $n$, the inclusion induced map
\begin{equation}
H^*_G(B_GU(n))\to H^*_G(B_GSU(n))
\end{equation}
is a surjection, and $\gene_1=\genc_{1,n-1}=0$ in $H^*_G(SU(n))$. There are more relations however; for example, if $n=2$ there is an additional relation $\genu^2=2\genu$ since $SU(2)=Sp(1)$.

In the stable case, we can distinguish between $B_G^+SU, B_G^-SU$ and $B_G^{\pm}SU$ and we have $\gene_1=\genc_1=0$.

	\section{\texorpdfstring{$C_2$ Chern classes}{C2 Chern classes}}\label{C2ChernSection}	
	
	The goal of this section is to prove our results on Chern classes. In effect, we need to prove the isomorphism:		
	\begin{gather}
	H^*_G(B_GU(n))=	(H^*_G(B_GU(1))^{\otimes n})^{\Sigma_n}
	\end{gather}
	and then use the computation of $H^*_G(B_GU(1))$ to obtain the algebraic description in Proposition \ref{C2Chern2} and Proposition \ref{C2ChernIso}.
	
	\subsection{\texorpdfstring{The $n=1$ computation}{The n=1 computation}}\label{n=1Comp}For any $C_2$-space $X$, by \cite{GM95},
	\begin{equation}
	H^*_{C_2}(X)=H^*(X)^{C_2}\oplus H^*(X^{C_2})
	\end{equation}
	Now take $X=B_GS^1$ which is $\C P^{\infty}$ with the $C_2$ action given on complex homogeneous coordinates by:
	\begin{equation}
	g(z_0:z_1:z_2:z_3:\cdots)=(z_0:-z_1:z_2:-z_3:\cdots)
	\end{equation}
	We have $H^*(\C P^{\infty})=\Q[r]$ for a generator $r$ of degree $2$. The $C_2$ action is trivial as can be verified on the $2$-skeleton $S^2=\C P^1\subseteq \C P^{\infty}$: the $C_2$ action on $S^2$ is a rotation hence has degree $1$.
	
	We also have $(B_GS^1)^{C_2}=\C P^{\infty}\coprod \C P^{\infty}$ spanned by $v^+=(z_0:0:z_2:0:\cdots)$ and $v^-=(0:z_1:0:z_3:\cdots)$ respectively. Thus we get
	\begin{equation}
	H^*_{C_2}(B_GS^`)=H^*(BS^1)\oplus H^*((B_GS^1)^{C_2})=\Q[\geneidem_1]\oplus \Q[\geneidem_2]\oplus \Q[\geneidem_3]
	\end{equation}
	where $e_1$ is the nonequivariant generator and $\geneidem_2,\geneidem_3$ correspond to $v^+,v^-$ respectively. 

	Define $\genu=\geneidem_2^0$ and $\genw=\geneidem_1+\geneidem_2+\geneidem_3$; then
	\begin{gather}
	\geneidem_1^0=x/2\text{ , }\geneidem_2^0=\genu\text{ , }\geneidem_3^0=y-\genu\\
	\geneidem_1=x\genw/2\text{ , } \geneidem_2=\genw\genu\text{ , }\geneidem_3=(y-\genu)\genw
	\end{gather}
	
	We have proven:	
	\begin{prop}\label{C2Chern1Class}As an algebra over $A_{\Q}$:
		\begin{equation}
		H_G^*(B_{C_2}S^1)=\frac{A_{\Q}[\genw,\genu]}{\genu^2=\genu, x\genu=0}
		\end{equation}
		for $|\genw|=2$ and $|\genu|=0$.
	\end{prop}
	
	We should compare this with the description $$H_G^{\bigstar}(B_GS^1)=H^{\bigstar}_G[c,b]/(c^2=a_{\sigma}^2c+u_{2\sigma}b)$$ obtained in \cite{Shu14}. The correspondence of generators is:
	\begin{gather}
	\genu=c\frac{y}{a_{\sigma}^2}\\
	\genw=c\frac{x}{2u_{2\sigma}}+b\frac{y}{a_{\sigma}^2}\\
	c=\genw u_{2\sigma}+\genu a_{\sigma}^2\\
	b=-\genw^2u_{2\sigma}+\genw a_{\sigma}^2
	\end{gather}
	
	\subsection{Maximal tori and Weyl groups}
	
	If $L$ is a compact connected Lie group and $T$ a maximal torus in $L$, we have the inclusion-induced map
	\begin{equation}
	H_G^*(B_GL)\to H_G^*(B_GT)
	\end{equation}
	The Weyl group $W=W_LT$ acts on $L, T$ by conjugation hence on $H^*(B_GT), H^*(B_GL)$ and the inclusion-induced map is $W$-equivariant. Actually, $W$ acts trivially on $B_GL$, which is a special case of the fact that an inner automorphism of $H$ induces the identity map on $B_GH$ up to homotopy (see \cite{BCM} for the classical nonequivariant case; the equivariant generalization is straightforward). Thus our map factors through the $W$-fixed points:
	\begin{equation}
	H_G^*(B_GL)\to H_G^*(B_GT)^W
	\end{equation}
	This breaks into:
	\begin{gather}
	H^*(BL)\to H^*(BT)^W\\
	H^*((B_GL)^{C_2})\to H^*((B_GT)^{C_2})^W
	\end{gather}
	The first map is an isomorphism (\cite{BCM}), so if we can prove that the second map is an isomorphism then 
	\begin{equation}
	H^*_G(B_GL)\to H^*_G(B_GT)^W
	\end{equation}
	will also be an isomorphism.

	If $T=\prod^nS^1$ then
	\begin{equation}
	(B_GT)^{C_2}=\coprod^{2^n}\prod^nBS^1  \end{equation}
	The coproduct is indexed over $C_2^n$ i.e. sign configurations $(\pm,...,\pm)$ (the next subsection explains why it's natural to use sign configurations). By considering the number $m$ of $+$'s in a configuration, we can further break this into
	\begin{equation}
	(B_GT)^{C_2}=\coprod_{m=0}^n\coprod^{\binom nm}\prod^nBS^1  \end{equation}
	
	In cohomology:
	\begin{equation}
	H^*((B_GT)^{C_2})=\oplus_{m=0}^n\oplus^{\binom nm}H^*(BS^1)^{\otimes n}  \end{equation}
	
	The $\Sigma_n$ action permuting the $S^1$ factors in $T$ has the effect of preserving the $m$, permuting the $\binom nm$ many sign configurations and permuting the factors in the tensor product.
	
	\subsection{The maximal torus isomorphism}
	
	We use the Grassmannian model $Gr(n,\C^{\infty\rho})$ for $B_GU(n)$, that consists of $n$-dimensional (complex) subspaces of $\C^{\infty\rho}$ where $\rho=1+\sigma$ is the complex regular representation of $G$. We have:
	\begin{equation} 
	B_GU(n)^{C_2}=\coprod_{m=0}^nBU(m)\times BU(n-m)
	\end{equation} 
	Indeed, a fixed point $V$ is a $C_2$ subspace and thus admits a unique decomposition $V=V^+\oplus V^-$ where $gv=v$ for any $v\in V^+$ and $gv=-v$ for any $v\in V^-$. So a subspace $V$ in the LHS corresponds to the pair $(V^+,V^-)$ in the RHS.
	% \begin{proof}Any vector in $\C^{\infty \rho}$ can be decomposed uniquely as $v=v^++v^-$ where $v^+$ has only even nonzero coordinates and $v^-$ has only odd nonzero coordinates; $C_2$ then acts as $gv^+=v^+$ and $gv^-=-v^-$. If $W$ is an $n$-dimensional subspace of $\C^{\infty\rho}$ fixed by the $C_2$ action, then it can be decomposed uniquely as $W=W^+\oplus W^-$ where $W^+,W^-$ only have vectors of the form $v^+,v^-$ respectively; if $m$ is the dimension of $W^+$ then $(W^+,W^-)$ defines an element of $BU(m)\times BU(n-m)$.\end{proof}
	
	We use the maximal torus $T=\prod^nS^1$ in $U(n)$. The Weyl group is $\Sigma_n$ and the action is by permuting the $S^1$ factors.  
	\begin{prop}The maximal torus inclusion induces an isomorphism:
		\begin{equation}
		H_G^*(B_GU(n))= H_G^*(B_GT)^{\Sigma_n}
		\end{equation}
	\end{prop}

	\begin{proof}The map $B_GT\to B_GU(n)$ is on the $C_2$-fixed points:
		\begin{equation} 
		\coprod^{2^n}BU(1)^n\to \coprod_{m=0}^nBU(m)\times BU(n-m)
		\end{equation} 
		sending $(v_1^{\pm},...,v_n^{\pm})$ to the direct sum $v_1^{\pm}\oplus ...\oplus v_n^{\pm}$ (we are implicitly using the identification $Gr_n(\C^{\infty \rho}\oplus \C^{\infty \rho})=Gr_n(\C^{\infty \rho})$ through a fixed linear equivariant isomorphism $\C^{\infty \rho}\oplus \C^{\infty \rho}=\C^{\infty\rho}$). Here, $v_i^+$ denotes a $1$-dimensional subspace with trivial $C_2$ action, while $v_i^-$ is a $1$-dimensional subspace with antipodal $C_2$ action. The signs in  $(v_1^{\pm},...,v_n^{\pm})$ correspond to the sign configuration and the index $m$ on the RHS corresponds to the amount of $+$ signs in a configuration. Thus, the map above breaks into
		\begin{equation} 
		\coprod^{\binom nm}BU(1)^n\to BU(m)\times BU(n-m)
		\end{equation} 		
		for every $m=0,...,n$. Fixing the $m$, the induced map on cohomology is
		\begin{equation} 
		H^*(BU(m))\otimes H^*(BU(n-m))\to \oplus^{\binom nm} H^*(BU(1))^{\otimes n}
		\end{equation} 
		The action of $\Sigma_n$ on the right permutes the factors in the tensor product and the sign configuration. Taking $\Sigma_n$ fixed points is equivalent to fixing our favorite configuration, say $(+,...,+,-,...,-)$, and then taking $\Sigma_m\times \Sigma_{n-m}$ fixed points, where $\Sigma_m$ permutes only the $+$'s and $\Sigma_{n-m}$ permutes only the $-$'s.
		%example: For $n=2, m=1$ we have polynomials on a^+,b^+,a^-,b^- where a^+a^-=b^+b^-=0 whose terms involve both +,- and Sigma_2 exchanges a^+,b^+ and a^-b^-. Any fixed point is p(a^+,b^-)+p(b^+,a^-) for the same polynomial p, hence is isomorphic to the ring Q[a,b]
		With that in mind, our map factors through $\Sigma_n$ fixed points and we get
		\begin{gather} 
		H^*(BU(m))\otimes H^*(BU(n-m))\to (H^*(BU(1))^{\otimes m})^{\Sigma_m}\otimes (H^*(BU(1))^{\otimes (n-m)})^{\Sigma_{n-m}}
		\end{gather} 
		This is the tensor product of maps $H^*(BU(i))\to (H^*(BU(1))^{\otimes i})^{\Sigma_i}$ where $i=m,n-m$. These maps are induced by the maximal torus inclusions $U(1)^i\to U(i)$, hence are isomorphisms by the nonequivariant case.
	\end{proof}

	\subsection{The fixed point computation}\label{TheAlgebraPart} 
	
The cohomology of the maximal torus is: \begin{equation}
H^*_G(B_GT)= H^*_G(\prod_nB_GS^1)=\frac{A_{\Q}[\genw_i,\genu_i]_{1\le i\le n}}{x\genu_i=0, \genu_i^2=\genu_i}
\end{equation}
with the $\Sigma_n$ action permuting the $\genw_i$ and $\genu_i$ separately (namely $\sigma \genw_i=\genw_{\sigma(i)}$ and $\sigma \genu_i=\genu_{\sigma(i)}$ for $\sigma\in \Sigma_n$). In this subsection, we summarize the computation of the $\Sigma_n$-fixed points:
\begin{equation}H^*_G(B_GU(n))=(H_G^*(B_GS^1)^{\otimes n})^{\Sigma_n}=\Big(\frac{A_{\Q}[\genw_i,\genu_i]}{x\genu_i=0, \genu_i^2=\genu_i}\Big)^{\Sigma_n}
\end{equation}
that is proven in a more general form in Appendix \ref{appen} (see Propositions \ref{AlgebraPropositionGeneral} and \ref{AlgebraProposition2General}). We have:
	\begin{cor}\label{AlgebraCorollaryQ}For $1\le i,s\le n$ and $1\le j\le n-s$ consider the elements of the $A_{\Q}$-algebra:
		\begin{equation}H_G^*(B_GS^1)^{\otimes n}=\frac{A_{\Q}[\genw_i,\genu_i]}{x\genu_i=0, \genu_i^2=\genu_i}	\end{equation}
		given by:
	\begin{gather}
	\genu=\sigma_1(\genu_1,...,\genu_n)=\sum_{1\le m\le n}\genu_m\\
	\gene_i=\sigma_i(\genw_1,...,\genw_n)=\sum_{m_*\in K_i}\genw_{m_1}\cdots \genw_{m_i}\\
	\genc_{s,j}=\sum_{(m_*,l_*)\in K_{s,j}}\genw_{m_1}\cdots \genw_{m_s}\genu_{l_1}\cdots \genu_{l_j}
	\end{gather}
	where $K_i$ consists of partitions $1\le m_1<\cdots<m_i\le n$ , $K_{s,j}\subseteq K_s\times K_j$ consists of pairs of disjoint partitions and $\sigma_i$ is the $i$-th elementary symmetric polynomial.\\
	Then:
	\begin{equation}
H^*_G(B_GU(n))=\frac{A_{\Q}[\genu,\gene_i,\genc_{s,j}]}{x\genu=0, x\genc_{s,j}=0, S}
	\end{equation}
	where the finite set $S$ of relations consists of three types of relations:
	\begin{itemize}
		\item  Type I: 
		$$\genu^{n+1}=\sum_{m=1}^nr_m\genu^{m}$$
		where $r_1=(-1)^nn!$ and
		\begin{equation}
		r_{m+1}=(-1)^{n+m}n!\sum_{1\le i_1<\cdots<i_m\le n}\frac1{i_1\cdots i_m}
		\end{equation}
		We can also write this relation as
		\begin{equation}
		\genu(\genu-1)\cdots (\genu-n)=0
		\end{equation}
		\item Type II:
		\begin{equation}
		\genu^s\genc_{s,i}=\frac{s!}{(s+i)!}\gene_s\genu^{s+i}+\cdots
		\end{equation} 
		where $\cdots$ denotes a homogeneous polynomial smaller than $\gene_s\genu^{s+i}$ (see Appendix \ref{appen} for a definition of this order).
		\item Type III: If  $s\le t\le s+i$,
		$$\genc_{s,i}\genc_{t,j}=\binom{\min(i+j+s,n)-t}{j}\gene_t\genc_{s,\min(i+j,n-s)}+\cdots$$
		where $\cdots$ denotes a homogeneous polynomial smaller than $\gene_t\genc_{s,\min(i+j,n-s)}$.
	\end{itemize}
	The polynomials $\cdots$ can be algorithmically computed in terms of $\genu,\gene_i,\genc_{s,j}$; the algorithm is described in Appendix \ref{appen} and has been implemented in the computer program found \href{https://github.com/NickG-Math/Symmetric_Polynomials}{here}.
	
	The elements $\gene_i$ are algebraically independent and $H^*_G(B_GU(n))$ is finite over $A_{\Q}[\gene_1,...,\gene_n]$ hence has Krull dimension $n$.
	
	A basis of $H^*_G(B_GU(n))$ over $A_{\Q}$ consists of the elements
	\begin{equation}
	\genu^a\prod_{i=1}^n \gene_i^{k_i}\prod_{s=1}^n\prod_{i=1}^{n-s} \genc_{s,i}^{\epsilon_{s,i}}
	\end{equation}
	where $0\le a\le n$ and $\epsilon_{s,i}=0,1$ are such that for any $\genc_{s,i}$ appearing in the product we must have $a<s$ and for any two factors $\genc_{s,i},\genc_{t,j}$ with $s\le t$ we must have $s+i<t$.\end{cor}
	
		\begin{cor}\label{AlgebraCorollary2Q}Any set of homogeneous algebra generators of $H^*_G(B_GU(n))$ over $A_{\Q}$ has cardinality at least $1+n+\binom n2$, which is the cardinality of the generating set $\{\genu,\gene_s,\genc_{s,j}\}$.
	\end{cor}

	We can extend $\genc_{s,i}$ to $s=0$ and $i=0$ via: 
\begin{align}
\genc_{s,0}&=y\gene_s=y\sigma_i(\genw_1,...,\genw_n)\\
\genc_{0,i}&=\sigma_i(\genu_1,...,\genu_n)=(i!)^{-1}\genu(\genu-1)\cdots (\genu-i+1)\\
\genc_{0,0}&=y
\end{align}
Whenever we write $\genc_{s,i}$ it is implicit that $s,i>0$, unless we explicitly state that we are using the convention above.

	\subsection{Dimension count}Consider the modified partition function $p(n,m)$ counting sequences $a_1\ge \cdots\ge a_n\ge 0$ with $m=\sum_ia_i$ (the usual partition function requires $a_n\ge 1$). We have the recursion
	\begin{equation}
	p(n,m)=p(n,m-n)+p(n-1,m)
	\end{equation}
	Then $p(n,m)$ is the dimension of the vector space of symmetric polynomials in $\Q[x_1,...,x_n]$ of degree $m$, which we grade as $|x_1|=\cdots=|x_n|=1$. If $R$ is as in subsection \ref{TheAlgebraPart} with $k=\Q$, we have:
	\begin{equation}
	\dim (R^{\Sigma_n}_m)=\sum_{i=0}^n\sum_{j=0}^mp(i,j)p(n-i,m-j)
	\end{equation}
	We can equivalently express these facts as:
	\begin{equation}
	\dim H^{2m}(BU(n))=p(n,m)
	\end{equation}
	and
	\begin{equation}
	\dim H^{2m}_G(B_GU(n))=p(n,m)+\sum_{i=0}^n\sum_{j=0}^mp(i,j)p(n-i,m-j)
	\end{equation}
	Note that for fixed $m$ the dimensions $\dim H^{2m}(BU(n))$ stabilize for large enough $n$. That is not the case for $\dim H^{2m}_GB_GU(n)$.

	\section{\texorpdfstring{$C_2$ Chern classes of sums and tensor products}{C2 Chern classes of sums and tensor products}}\label{C2ChernSumTensor}
	
	\begin{prop}The map $B_GU(n)\to B_GU(n+1)$ given by direct sum with a trivial complex representation induces on cohomology:
		\begin{gather}
		\genu\mapsto y+\genu\\
		\gene_i\mapsto \gene_i\\
		\genc_{s,i}\mapsto \genc_{s,i}+\genc_{s,i-1}
		\end{gather}
		using the convention $\gamma_{s,0}=yc_s$.\\
		The map $B_GU(n)\to B_GU(n+1)$ given by direct sum with a complex $\sigma$ representation induces on cohomology:
		\begin{gather}
		\genu\mapsto \genu\\
		\gene_i\mapsto \gene_i\\
		\genc_{s,i}\mapsto \genc_{s,i}
		\end{gather}
		For both maps we use the conventions that $\gene_{n+1}=0$ and $\genc_{s,n+1-s}=0$ in every RHS.
	\end{prop}	
	\begin{proof}We have a commutative diagram
		\begin{center}\begin{tikzcd}
			B_GU(n)\ar[r,"\oplus 1"]&B_GU(n+1)\\
			B_GT^n=B_GT^n\times *\ar[u]\ar[r]&B_GT^n\times B_GS^1=B_GT^{n+1}\ar[u]
			\end{tikzcd}\end{center}
		where the bottom map is the product of the identity map $B_GT^n\to B_GT^n$ and the inclusion map $*\to B_GS^1$ given by $*\mapsto v^+$ where $v^+=(1:0:1:0:\cdots)$ in the homogeneous coordinates of $B_GS^1$. This inclusion map induces on cohomology: $$A_{\Q}[\genu,\genw]/(\genu^2=\genu, x\genu) \to A_{\Q}$$ given by $y-\genu\mapsto 0$ and $\genw\mapsto 0$ (this is verified by looking at the $C_2$ fixed points). Similarly, adding a $\sigma$ representation induces $$A_{\Q}[\genu,\genw]/(\genu^2=\genu, x\genu) \to A_{\Q}$$ given by
		$\genu\mapsto 0$ and $\genw\mapsto 0$.	Thus the $\oplus 1$ induced map is determined by:
		\begin{gather}
		\genw_i\mapsto \genw_i\text{ , }\genu_i\mapsto \genu_i\text{ , }i<n+1\\
		\genw_{n+1}\mapsto 0\text{ , }\genu_{n+1}\mapsto y
		\end{gather}
		The $\oplus \sigma$ induced map is determined by:
		\begin{gather}
		\genw_i\mapsto \genw_i\text{ , }\genu_i\mapsto \genu_i\text{ , }i<n+1\\
		\genw_{n+1}\mapsto 0\text{ , }\genu_{n+1}\mapsto 0
		\end{gather}
		These descriptions imply the ones on the generators $\genu,\gene_i,\genc_{s,j}$.\end{proof}
	
	\begin{prop}
		The direct sum of bundles map $B_GU(n)\times B_GU(m)\to B_GU(n+m)$ induces on cohomology: 
		\begin{gather}
		\genu\mapsto \genu\otimes 1+1\otimes \genu\\
		\gene_i\mapsto \sum_{j+k=i}\gene_j\otimes \gene_k\\
		\genc_{s,i}\mapsto \sum_{s'+s''=s\atop i'+i''=i}\genc_{s',i'}\otimes \genc_{s'',i''}
		\end{gather} 
		using the convention for defining $\gene_0,\genc_{s,0},\genc_{0,i}$ in the RHS.
	\end{prop}
	\begin{proof}The corresponding map on maximal tori $$\prod^nB_GS^1\times \prod^mB_GS^1\to \prod^{n+m}B_GS^1$$ induces on cohomology:
		\begin{equation}
		\genw_i\mapsto \begin{cases}
		\genw_i\otimes 1&\textup{ if }i\le n\\
		1\otimes \genw_{i-n}&\textup{ if }i>n
		\end{cases}\text{ , }
		\genu_i\mapsto \begin{cases}
		\genu_i\otimes 1&\textup{ if }i\le n\\
		1\otimes \genu_{i-n}&\textup{ if }i>n
		\end{cases}
		\end{equation}
		This implies the formulas on $\genu,\gene_i,\genc_{s,j}$.
	\end{proof}
	
	\begin{prop}\label{TensorProof}
		
		The tensor product of bundles map $B_GU(1)\times B_GU(1)\to B_GU(1)$ induces on cohomology:
		\begin{gather}
		\genu\mapsto y-\genu\otimes 1-1\otimes \genu+2\genu\otimes \genu\\
		\genw\mapsto \genw\otimes 1+1\otimes \genw\end{gather}
	\end{prop}
	
	\begin{proof}The map in question, induced from multiplication $S^1\times S^1\to S^1$, is given on the homogeneous coordinates by multiplication of polynomials:
		\begin{equation}
		(x_0:x_1:\cdots)\otimes (y_0:y_1:\cdots)=(x_0y_0:x_0y_1+x_1y_0:\cdots)
		\end{equation}
	Write
		\begin{equation}
		H^*_{C_2}(B_GU(1))=H^*(BU(1))\oplus H^*(B_GU(1)^{C_2})=\Q[\geneidem_1]\oplus \Q[\geneidem_2]\oplus \Q[\geneidem_3]
		\end{equation}
		as in subsection \ref{n=1Comp}.  Then $B_GU(1)\times B_GU(1)\to B_GU(1)$ induces
		\begin{equation}
		H^*_G(B_GU(1))\to H^*_G(B_GU(1))\boxtimes_{A_{\Q}} H^*_G(B_GU(1))
		\end{equation}
		which breaks into
		\begin{gather}
		H^*(BU(1))\to H^*(B_GU(1))\otimes H^*(B_GU(1))\\
		H^*(B_GU(1)^{C_2})\to H^*(B_GU(1)^{C_2})\otimes H^*(B_GU(1)^{C_2})
		\end{gather}
		By the nonequivariant case, the first map is
		\begin{gather}
		\geneidem_1^0\mapsto \geneidem_1^0\otimes \geneidem_1^0\\
		\geneidem_1\mapsto \geneidem_1\otimes \geneidem_1^0+\geneidem_1^0\otimes \geneidem_1
		\end{gather}
		For the second map, note that in the $H$-space structure, $v^+,v^-$ multiply according to: $v^{\alpha}\cdot v^{\beta}=v^{\alpha\beta}$ for $\alpha,\beta=\pm 1$. This means that the $+$ part of $H^*(B_GU(1)^{C_2})$ maps to the $+\otimes +$ and $-\otimes -$ parts of $H^*(B_GU(1)^{C_2})\otimes H^*(B_GU(1)^{C_2})$ to give:
		\begin{gather}
		\geneidem_2^0\mapsto \geneidem_2^0\otimes \geneidem_2^0+\geneidem_3^0\otimes \geneidem_3^0\\
		\geneidem_2\mapsto \geneidem_2\otimes \geneidem_2^0+\geneidem_2^0\otimes \geneidem_2+\geneidem_3\otimes \geneidem_3^0+\geneidem_3^0\otimes \geneidem_3
		\end{gather}
		Similarly, the $-$ part of $H^*(B_GU(1)^{C_2})$ maps to the $+\otimes -$ and $-\otimes +$ parts of $H^*(B_GU(1)^{C_2})\otimes H^*(B_GU(1)^{C_2})$ to give:
		\begin{gather}
		\geneidem_3^0\mapsto \geneidem_2^0\otimes \geneidem_3^0+\geneidem_3^0\otimes \geneidem_2^0\\
		\geneidem_3\mapsto \geneidem_2\otimes \geneidem_3^0+\geneidem_2^0\otimes \geneidem_3+\geneidem_3\otimes \geneidem_2^0+\geneidem_3^0\otimes \geneidem_2
		\end{gather}
		In terms of the $\genw,\genu$ generators, recall $\geneidem_2^0=\genu, \geneidem_3^0=y-\genu, \geneidem_1=\genw\geneidem_1^0, \geneidem_2=\genw\geneidem_2^0$ and $\geneidem_3=\genw\geneidem_3^0$. Substituting these gives the desired formulas.
		%As expected, the map $H^{\bigstar}(B_GU(1))\to H^{\bigstar}(B_GU(1)\times B_GU(1))$ factors through $H^{\bigstar}(B_GU(2))$.
	\end{proof}

Iterating $B_GS^1\times B_GS^1\to B_GS^1$ gives $(B_GS^1)^{\times n}\to B_GS^1$ which induces on cohomology:
\begin{gather}
\genu\mapsto \frac{(-1)^n+1}2y+(-1)^{n+1}\sum_{i=1}^n(-2)^{i-1}\sigma_i(\genu_1,...,\genu_n)\\
\genw\mapsto \genw_1+\cdots+\genw_n
\end{gather}
We will also need that the map induced from conjugation $B_GU(1)\to B_GU(1)$ is:
	\begin{gather}
	\genu\mapsto \genu\\
	\genw\mapsto -\genw
	\end{gather}
This is verified similarly to Proposition \ref{TensorProof}.
	
	\section{\texorpdfstring{$C_2$ stable Chern classes}{C2 stable Chern classes}}\label{C2ChernStable}
	
	Since there are two maps $B_GU(n)\to B_GU(n+1)$ (given by direct sum with the trivial or the $\sigma$ representation) one can try to stabilize against, there are a few distinct notions of stable characteristic classes.
	
	\subsection{Stabilizing against one representation} First, we can stabilize with respect to direct sum with the trivial representation and get:
	\begin{equation}
	B_G^+U=\colimit(B_GU(1)\xrightarrow{\oplus 1}B_GU(2)\xrightarrow{\oplus 1}\cdots)
	\end{equation}
	Then
	\begin{equation}
	K_G^+(X)=[X,B_G^+U\times \Z]^{G}
	\end{equation}
	is the semi-ring of virtual bundles $V-n$ on $X$ modulo the equivalence relation $V\sim W\iff V+n=W+n$ for some $n$. For example, $K_G^+(*)=\Z\times \mathbb N$. Group completing $K_G^+$ gives equivariant $K$-theory $K_G$.
	
	The fixed points of $B_G^+U$ are:
	\begin{equation}
	(B_G^+U)^{C_2}=BU\times \coprod_{n=0}^{\infty}BU(n)
	\end{equation}
%	under $V\mapsto (V^+,V^-)$ in the Grassmannian model. 
 We note that $B_G^+U$ is the usual equivariant classifying space $B_GU=E_GU/U$.
	
	An equivalent way of getting $B_GU$ is to stabilize with respect to direct sum with the $\sigma$ representation:
	\begin{equation}
	B_G^-U=\colimit(B_GU(1)\xrightarrow{\oplus \sigma}B_GU(2)\xrightarrow{\oplus \sigma}\cdots)
	\end{equation}
	Note that 	$B_G^+U\simeq B_G^-U$.
	
	\subsection{Stabilizing with respect to all representations} Group completing $B_G^+U$ or equivalently $B_G^-U$, gives:
	\begin{equation}
	B_G^{\pm}U=\colimit(B_G^+U\xrightarrow{\oplus \sigma}B_G^+U\xrightarrow{\oplus \sigma}\cdots)=\colimit(B_G^-U\xrightarrow{\oplus 1}B_G^-U\xrightarrow{\oplus 1}\cdots)
	\end{equation}
		Then
	\begin{equation}
	K_G(X)=[X,B_G^{\pm}U\times \Z]^{G}
	\end{equation}
	is the semi-ring of virtual bundles  $V-n-m\sigma$ modulo the equivalence relation $V\sim W\iff V+n+m\sigma=W+n+m\sigma$ for some $n,m$. Thus $K_G(X)$ is exactly the equivariant $K$-theory of $X$.
	
	We finally have the fixed point computation:
	\begin{equation}
	(B_G^{\pm}U)^{C_2}=BU\times BU\times \Z
	\end{equation}
%	under
%	\begin{equation}
%	V-n\sigma\mapsto (V^+,V^-,|V^-|-n)
%	\end{equation}
%	in the Grassmannian model.

	\subsection{Cohomology computations}
	
	We prefer to work with $B_G^-U$ since the map\\
	$H^*_G(B_GU(n+1))\xrightarrow{\oplus \sigma} H^*_G(B_GU(n))$ has a simpler expression on the generators $\genu,\genc_{s,j}$ compared to the map $H^*_G(B_GU(n+1))\xrightarrow{\oplus 1} H^*_G(B_GU(n))$. 
	
	The $\Q$-subalgebras spanned by $\genu$ in each $H_G^*(B_GU(n))$ have limit:
	\begin{gather}
	\limit \frac{\Q[\genu]}{\genu(\genu-1)(\genu-2)\cdots (\genu-n)}= \prod_{n\ge 0}\Q\end{gather}
under the correspondence:
\begin{gather}
a_0+\sum_{i\ge 0}a_{i+1}\genu(\genu-1)\cdots (\genu-i)\mapsto (a_0,a_0+a_1,a_0+2a_1+2a_2,...)
	\end{gather}
given by evaluating the series in the LHS at $\genu=0,1,2,...$.
% For example, the series $1+\sum_{n\ge 0} \frac{(-1)^{n+1}}{n!}\genu(\genu-1)\cdots(\genu-n)$ maps to $(1,0,...)$.

Under this correspondence we have:
	\begin{equation}
	H^0_G(B_G^-U)=A_{\Q}\times \prod_{n\ge 1}A_{\Q}/x
\end{equation}
As a graded  $H^0_G(B_G^-U)[\gene_1,\gene_2,...]$-algebra, $H^*_G(B_G^-U)$ is generated by the series $\sum_{i=1}^{\infty}r_i\genc_{s,i}\in H^{2s}_G(B_G^-U)$ for $r_i\in \Q$ and $s=1,2,...$.\medbreak
	
	Describing $H^*_G(B_G^{\pm}U)$ in terms of the generators $\genu,\gene_i,\genc_{s,j}$ is even more complicated, as we need to take the limit of $H^*_G(B_G^-U)$ with respect to the $\oplus 1$ maps. We can alternatively view $H^*_G(B_G^{\pm}U)$ as the limit of a diagram indexed on $\mathbb N\times \mathbb N$ with $(n,m)\mapsto B_GU(n)$, horizontal maps being $\oplus 1$ and vertical maps being $\oplus \sigma$. An element of $H^*_G(B_G^{\pm}U)$ will then be a compatible doubly indexed sequence $s_{n,m}\in H^*_G(B_GU(n))$. The constant sequences (in both variables) consist of elements that are invariant under both the $\oplus 1$ and $\oplus \sigma$ maps. We can see that
	\begin{equation}
		\gene_i, \genc_i:=\gene_i\genu-\genc_{i,1}
	\end{equation}
have this property, and we conjecture that the subalgebra of constant sequences is generated by them. This is equivalent to the $\genc_1,\genc_2,...$ being algebraically independent over $\Q[\gene_1,\gene_2,...]$ and further equivalent to the subalgebra of constant sequences being
	\begin{equation}
	\frac{A_{\Q}[\gene_i,\genc_i]}{x\genc_i} 
\end{equation}
The coalgebra structure is
\begin{gather}
	\genc_s\mapsto \sum_{i+j=s}(\gene_i\otimes \genc_j+\genc_i\otimes \gene_j)
\end{gather} 
using the conventions $\gene_0=1$ and $\genc_0=0$.

	\subsection{Homology computations}
	
	For homology, the idempotent decomposition gives
	\begin{equation}
	H_*^{C_2}(X)=H_*(X)^{C_2}\oplus H_*(X^{C_2})
	\end{equation}
	as Mackey functors. For an $H$-space $X$ where $X\times X\to X$ and $S^0\to X$ are $C_2$ equivariant, this becomes an isomorphism of Green functors. 
	
	Setting $X=B_G^-U$ and using 	\begin{equation}
	(B_G^-U)^{C_2}=\coprod_{n=0}^{\infty}BU(n)\times BU
	\end{equation}
	we get that $H_*^G(X)$ is the sum of
	\begin{gather}
	H_*(BU)=\Q[b_i^e]_{i\ge 1}\text{ , }\\
	H_*(B_G^-U^{C_2})= \oplus_{n\ge 1}H_*(BU(n)))\otimes H_*(BU)=\Q[b_0^+,b_i^+,b_i^-]_{i\ge 1}
	\end{gather}
where for $i\ge 1$, the $b_i^e,b_i^+,b_i^-$ are duals of $c_1^i$ using the first Chern class in the respective nonequivariant cohomology rings ($b_0^+$ indexes the components).

	If we let $b_i=b_i^e+b_i^++b_i^-$, $a_i=b_i^+$ for $i\ge 1$ and $d=x/2+b_0^+$ we get
	\begin{equation}
	H_*^G(B_G^-U)=A_{\Q}[d,a_i,b_i]_{i\ge 1}/(xd=x, xa_i=0)
	\end{equation}
	Group completing is localization at $d$:
	\begin{equation}
	H_*^G(B_G^{\pm}U)=A_{\Q}[d^{\pm},a_i,b_i]_{i\ge 1}/(xd=x, xa_i=0)
	\end{equation}
	%the map is $b_i^e\mapsto b_i^e$, $b_i^+\mapsto b_i^+b_0^+$ and $b_i^-\mapsto b_i^-b_0^+$ (so $1\mapsto d$; we have a module map, which is why $b_i^-\mapsto b_i^-b_0^+$ instead of $b_i^-\mapsto b_i^-$)
	
	The $a_i,b_i$ are dual to $\genc_1^i,c_1^i$ respectively, while $d$ is dual to the finite series \\$x/2+\genu\in H^0(B_G^-U)$ i.e. the sequence $(x/2,1,2,...)\in A_{\Q}\times \prod_{n\ge 1}A_{\Q}/x$.
	
	The coalgebra structure is:
	\begin{gather}
	d\mapsto d\otimes d\\
	a_i\mapsto \sum_{j+k=i}a_j\otimes a_k\\
	b_i\mapsto \sum_{j+k=i}(b_j\otimes b_k-b_j\otimes a_k-b_k\otimes a_j+2a_j\otimes a_k)
\end{gather}
using the conventions $a_0=d-x/2$ and $b_0=1$.

	\section{\texorpdfstring{$C_2$ symplectic classes}{C2 symplectic classes}}\label{C2Symplectic}
	
	Analogously to the Chern classes, we have:

	\begin{prop}There exist classes $\genu,k_i,\kappa_{s,j}\in H_G^*(B_GSp(n))$ of degrees $0,4i,4s$ respectively, where $1\le i,s\le n$ and $1\le j\le n-s$, such that
		\begin{equation}
		H_G^*(B_GSp(n))=A_{\Q}[\genu,k_i,\kappa_{s,j}]/(x\genu, x\kappa_{s,j}, S)
		\end{equation}
		The relation set $S$ is the same as that for $H_G^*(B_GU(n))$ with $\gene_i,\genc_{s,i}$ replaced by $k_i,\kappa_{s,i}$.\smallbreak	
		
		The generators $\genu, \kappa_{s,j}$ restrict to $0$ while the $k_i$ restrict to the nonequivariant symplectic classes.\smallbreak
		
		The maximal torus inclusion $U(1)^n\hookrightarrow Sp(n)$ induces an isomorphism
		\begin{gather}
		H^*_G(B_GSp(n))=	(H^*_G(B_GU(1))^{\otimes n})^{C_2\wr \Sigma_n}
	\end{gather}
Explicitly:
\begin{gather}
		A_{\Q}[\genu,k_i,\kappa_{s,j}]/(x\genu, x\kappa_{s,j}, S)=(A_{\Q}[\genu_i,\genw_i]/(x\genu_i))^{C_2\wr\Sigma_n}
		\end{gather}
Under this identification:
		\begin{gather}
		\genu=\sum_{1\le i\le n}\genu_i\\
		k_i=\sum_{m_*\in K_i}\genw_{m_1}^2\cdots \genw_{m_i}^2\\
		\kappa_{s,j}=\sum_{(m_*,l_*)\in K_{s,j}}\genw_{m_1}^2\cdots \genw_{m_s}^2\genu_{l_1}\cdots \genu_{l_j}
		\end{gather}
		where $K_i$ and $K_{s,j}$ are as in Corollary \ref{AlgebraCorollaryQ}. 
	\end{prop}
	
	\begin{proof}First, we have the fixed point computation:
		\begin{equation}
		B_GSp(n)^{C_2}=\coprod_{m+k=n}BSp(m)\times BSp(k)
		\end{equation}
		The maximal torus in $Sp(n)$ is $T=\prod^nS^1$, the same as in $U(n)$, but the Weyl group now is $C_2\wr \Sigma_n$ with $\Sigma_n$ permuting the $S^1$ factors and the $i$-th $C_2$ in $C_2\wr \Sigma_n=C_2^n\rtimes \Sigma_n$ acting as conjugation on the $i$-th $S^1$ factor in $T$.
		
		Following the $B_GU(n)$ case, the maximal torus inclusion on the fixed points breaks into
		\begin{equation} 
		\coprod^{\binom nm}BU(1)^n\to BSp(m)\times BSp(n-m)
		\end{equation} 
		where the coproduct of the left is indexed on sign configurations with $m$ many $+$'s. The induced map on cohomology is
		\begin{equation} 
		H^*(BSp(m))\otimes H^*(BSp(n-m))\to \oplus^{\binom nm} H^*(BU(1))^{\otimes n}
		\end{equation} 
		The action of $\Sigma_n$ on the right permutes the factors in the tensor product and the sign configuration, while the $i$-th $C_2$ in $C_2^n\rtimes \Sigma_n$ acts by as the conjugation-induced map $H^*(BS^1)\to H^*(BS^1)$ on the $i$-th factor in the tensor product (fixing the sign configuration). 
		
		Thus analogously to the $B_GU(n)$ case, we have an isomorphism into the $C_2\wr \Sigma_n$ fixed points of the right hand side. In conclusion:
		\begin{equation}
		H_G^*(B_GSp(n))=(\otimes^n H_G^*(B_GS^1))^{C_2\wr \Sigma_n}
		\end{equation}

		To compute the $C_2$ action on $H^*_G(B_GS^1)$ recall that conjugation $S^1\to S^1$ induces $\genu\mapsto \genu$ and $\genw\mapsto -\genw$ on cohomology. Fixing under the $C_2^n$ action, we get that $H_G^*(B_GSp(n))=R'^{\Sigma_n}$ where $R'=R(\genw_1^2,...,\genw_n^2,\genu_1,...,\genu_n)$ is the ring $R$ we used for $B_GU(n)$ but now with $\genw_i$ replaced by $\genw_i^2$.  So we get equivariant symplectic classes $\genu, k_i, \kappa_{s,i}$ of degrees $0,4i,4s$ with the desired expressions in terms of $\genu_i,\genw_i^2$.
	\end{proof}

	The maps $B_GSp(n)\to B_GSp(n+1)$ and $B_GSp(n)\times B_GSp(m)\to B_GSp(n+m)$ have the same formulas as the analogous maps for $B_GU(n)$, with $\gene_i,\genc_{s,i}$ replaced by $k_i,\kappa_{s,i}$ respectively.

	\begin{prop}\label{QuaterExplained}The forgetful map $B_GSp(n)\to B_GU(2n)$ induces on cohomology:
		\begin{gather}
		\genu\mapsto \genu\\
		\gene_{2i+1}, \genc_{2s+1,j}\mapsto 0\\
		\gene_{2i}\mapsto (-1)^ik_i\\
		\genc_{2s,j}\mapsto (-1)^s\kappa_{s,j}
		\end{gather}
		The quaternionization map $B_GU(n)\to B_GSp(n)$ induces:
		\begin{gather}
		\genu\mapsto \genu\\
		k_i\mapsto \sum_{a+b=2i}(-1)^{a+i}\gene_a\gene_b\\
		\kappa_{1,j}\mapsto \gene_1\genc_{1,j}-\genu\genc_{2,j-1}+(j-2)\genc_{2,j}+(j-1)\genc_{2,j-1}\\
		\kappa_{s,j}\mapsto \gene_s\genc_{s,j}+\cdots
		\end{gather}
	where $\cdots$ denotes a homogeneous polynomial in $R^{\Sigma_n}$ smaller than $\gene_s\genc_{s,j}$ (according to the order defined in appendix \ref{appen}). This polynomial can be computed algorithmically according to the algorithm in appendix \ref{appen} which has been implemented in the computer program found \href{https://github.com/NickG-Math/Symmetric_Polynomials}{here}.
		
	\end{prop}
	\begin{proof}
		To compute the effect of the forgetful map $B_GSp(n)\to B_GU(2n)$ note that we have the commutative diagram
		\begin{center}\begin{tikzcd}
			T^n\ar[r,"\text{$z_i\mapsto (z_i,\bar z_i)$}"]\ar[d]&T^{2n}\ar[d]\\
			Sp(n)\ar[r]&U(2n)
			\end{tikzcd}\end{center}
		(here $T=S^1$). The top map induces $H^*_G(B_GT^{2n})\to H^*_G(B_GT^n)$ given by $\genu_i\mapsto \genu_i, \genu_{i+n}\mapsto \genu_i, \genw_i\mapsto \genw_i, \genw_{i+n}\mapsto -\genw_i$ for $1\le i\le n$.
	\end{proof}
	
	The stable symplectic classes work analogously to the stable Chern classes; see subsection \ref{StableSummary} for a summary.

	\section{\texorpdfstring{$C_2$ Euler and Pontryagin classes}{C2 Euler and Pontryagin classes}}\label{C2EulerPontryagin}

	Analogously to the symplectic classes, we have:
	\begin{prop}There exist classes $\genu,p_i,\pi_{s,j},\chi$ of degrees $0,4i,4s,n$ respectively in $H^*_G(B_GSO(2n))$, where $1\le i,s< n$ and $1\le j\le n-s$ such that
		\begin{equation}
		H_G^*(B_GSO(2n))=A_{\Q}[\genu,p_i,\pi_{s,j},\chi]/(x\genu, x\pi_{s,j}, S)
		\end{equation}
		where the relation set $S$ is the same as that for $H_G^*(B_GU(n))$ with $\gene_i,\genc_{s,i}$ replaced by $p_i,\pi_{s,i}$ and using the convention $p_n=\chi^2$.\smallbreak	
		
		The generators $\genu,\pi_{s,j}$ restrict to $0$ while the $p_i,\chi$ restrict to the nonequivariant Pontryagin and Euler classes respectively.\smallbreak
		
		The map $B_GSO(2n)\to B_GSO(2n+1)$ induces an injection in cohomology with
		\begin{equation}
		H_G^*(B_GSO(2n+1))=A_{\Q}[\genu,p_i,\pi_{s,j}]/(x\genu, x\pi_{s,j}, S)
		\end{equation}
		where $i$ is allowed to be $n$ (i.e. $p_n=\chi^2$ is included). \smallbreak
		
		The maximal torus inclusion $T\hookrightarrow SO(n)$ induces an isomorphism
		\begin{gather}
		H^*_G(B_GSO(n))=(H^*_G(B_GT))^W
		\end{gather}
		where $W$ is the corresponding Weyl group. Under this isomorphism,
		\begin{gather}
		\genu=\sum_{1\le m\le n}\genu_m\\
		p_i=\sum_{m_*\in K_i}\genw_{m_1}^2\cdots \genw_{m_i}^2\\
		\pi_{s,j}=\sum_{(m_*,l_*)\in K_{s,j}}\genw_{m_1}^2\cdots \genw_{m_s}^2\genu_{l_1}\cdots \genu_{l_j}\\
		\chi=\genw_1\cdots \genw_n
		\end{gather}
		where $K_i, K_{s,j}$ are as in Corollary \ref{AlgebraCorollaryQ}. 
	\end{prop}

	\begin{proof}We use the oriented Grassmannian model for $B_GSO(n)$, consisting of $n$-dimensional oriented subspaces of $\R^{\infty\rho}$ where $\rho=1+\sigma$ is the real regular representation of $G=C_2$. The $G$ action sends an oriented subspace with basis $v_1,...,v_n$ to one with basis $gv_1,...,gv_n$. If $V\in B_GSO(n)$ is fixed by the $G$ action, then $V=V^+\oplus V^-$ where $V^+$ is an oriented subspace with $gv=v$ for every $v\in V^+$ and $V^-$ is an oriented subspace with $gv=-v$ for every $v\in V^-$. Moreover, since $g$ must act by an $SO(n)$ action on $V$, the dimension of $V^-$ must be even. In particular, if a $2$-dimensional subspace $V$ is fixed then $V=V^+$ or $V=V^-$. As for the uniqueness of the decomposition, note that if $W^+\oplus W^-= V^+\oplus V^-$ through an $SO(n)$ matrix $A$, then $A$ is block diagonal with blocks $B\in O(n-2k), C\in O(2k)$ and $\det(B)\det(C)=1$; $B$ gives $W^+= V^+$ while $C$ gives $W^-= V^-$. Thus
		\begin{equation}
		B_GSO(n)^{C_2}=\coprod_{k=0}^{n/2}BZ_{2k,n}
		\end{equation}
		where $Z_{k,n}$ is the subgroup of $O(k)\times O(n-k)$ consisting of pairs $(A,B)$ with $\det(A)\det(B)=1$. Note $Z_{k,n}=Z_{n-k,n}$ and $Z_{0,n}=Z_{n,0}=SO(n)$; if $0<k<n$ we have
		\begin{equation}
		Z_{k,n}=(SO(k)\times SO(n-k))\rtimes C_2
		\end{equation}
		with $C_2$ acting diagonally by conjugation.\smallbreak

		The maximal torus of $Z_{2k,2n}$ is $SO(2)^n$ with Weyl group $H\rtimes (\Sigma_k\times \Sigma_{n-k})\subseteq C_2^n\rtimes (\Sigma_k\times \Sigma_{n-k})$ where $H\subseteq C_2^n$ consists of elements with even number of coordinates equal to $-1$.\smallbreak
		
		The maximal torus of $Z_{2k,2n+1}$ is $SO(2)^n$ with Weyl group $C_2^n\rtimes (\Sigma_k\times \Sigma_{n-k})$.\smallbreak
		
		Following the $B_GU(n)$ case, the maximal torus inclusion $SO(2)^n\to SO(2n+1)$ induces $B_GSO(2)^n\to B_GSO(2n+1)$ which on the fixed points becomes:
		\begin{equation}
		\coprod^{2^{n}}BSO(2)^{n}\to \coprod_{k=0}^{n}BZ_{2k,2n+1}
		\end{equation}	
		with the LHS indexed over sign configurations as usual. Fixing the total amount $k$ of $-$ signs, the coproduct breaks into
		\begin{equation}
		\coprod^{\binom nk}BSO(2)^n\to BZ_{2k,2n+1}
		\end{equation}
		which induces
		\begin{equation}
		H^*(BZ_{2k,2n+1})\to \oplus^{\binom nk}H^*(BSO(2))^{\otimes n}
		\end{equation}
		Taking $C_2\wr \Sigma_n=C_2^n\rtimes \Sigma_n$ fixed points on the right hand side is equivalent to fixing a sign configuration and then taking $C_2^n\rtimes (\Sigma_k\times \Sigma_{n-k})$ fixed points i.e. fixing under the action of the Weyl group of $BZ_{2k,2n+1}$. Therefore we are reduced to proving $H^*(BZ_{2k,2n+1})=H^*(BSO(2)^n)^{C_2\wr (\Sigma_k\times \Sigma_{n-k})}$. Although Borel's Theorem is stated for connected Lie groups, it nonetheless works for $Z_{2k,2n+1}$ as it does for $O(n)$. This can be seen from the covering space
		\begin{equation}
			C_2\to BSO(2k)\times BSO(2n+1-2k)\to BZ_{2k,2n+1}
		\end{equation}
		for which the associated transfer map $$H^*(BZ_{2k,2n+1})\to H^*(BSO(2k)\times BSO(2n+1-2k))^{C_2}$$ is an isomorphism. Computing the right hand side shows that it's isomorphic to $H^*(BSO(2)^n)^{C_2\wr (\Sigma_k\times \Sigma_{n-k})}$ as desired.
		
		Similar arguments work for $H^*_G(B_GSO(2n))$ proving the maximal torus isomorphism in that case as well.
		
		Given the maximal torus isomorphism, we compute $H_G^*(B_GSO(2n+1))$ identically to $H_G^*(B_GSp(n))$ so we get the desired equivariant Pontryagin classes $p_i, \pi_{s,j}$. For $H_G^*(B_GSO(2n))$ we only want to fix under an even number of sign changes on the $\genw_i$ hence we get the equivariant Pontryagin classes plus the Euler class $\chi=\sigma_n(\genw_1,...,\genw_n)$. Note that $p_n=\chi^2$ and after removing $p_n$ there are no relations involving $\chi$ and the other generators.
	\end{proof}
	
	In the identification $B_GSO(2)=B_GU(1)$ we have
	\begin{gather}
	\chi=\gene_1
	\end{gather}

	The maps $B_GSO(n)\times B_GSO(m)\to B_GSO(n+m)$ work analogously to the symplectic case, replacing the $k_i,\kappa_{s,j}$ with $p_i,\pi_{s,j}$; the action on the Euler class is
	\begin{equation}
	\chi\mapsto \chi\otimes\chi
	\end{equation}
	
	\begin{prop}\label{ForgetExplained}The complexification map $B_GSO(2n)\to B_GU(2n)$ induces on cohomology:
		\begin{gather}
		\genu\mapsto \genu\\
		\gene_{2i+1}, \genc_{2s+1,j}\mapsto 0\\
		\gene_{2i}\mapsto (-1)^ip_i\\
		\genc_{2s,j}\mapsto (-1)^s\pi_{s,j}
		\end{gather}
		The forgetful map $B_GU(n)\to B_GSO(2n)$ induces on cohomology:
		\begin{gather}
		\genu\mapsto \genu\\
		p_i\mapsto \sum_{a+b=2i}(-1)^{a+i}\gene_a\gene_b\\
		\chi\mapsto \gene_n\\
		\pi_{1,j}\mapsto \gene_1\genc_{1,j}-u\genc_{2,j-1}+(j-2)\genc_{2,j}+(j-1)\genc_{2,j-1}\\
		\pi_{s,j}\mapsto \gene_s\genc_{s,j}+\cdots
		\end{gather}
		where $\cdots$ denotes the same homogeneous polynomial as the $\cdots$ in Proposition \ref{QuaterExplained}.
	\end{prop}

	\begin{proof}To understand the effect of complexification $SO(2n)\to U(2n)$ we use the nonstandard maximal torus $T^2\to U(2)$ making the following diagram commute:
		\begin{center}
			\begin{tikzcd}
			S^1\ar[d,equals]\ar[r,"\text{$a\mapsto (a,\bar a)$}"]&T^2\ar[d]\\
			SO(2)\ar[r]&U(2)
			\end{tikzcd}
		\end{center}
	where the map in the bottom row is complexification. More generally we have the commutative diagram
		\begin{center}
			\begin{tikzcd}[column sep=10em]
			T^n\ar[d]\ar[r,"\text{$(a_1,...,a_n)\mapsto (a_1,\bar a_1,...,a_n,\bar a_n)$}"]&T^{2n}\ar[d]\\
			SO(2n)\ar[r]&U(2n)
			\end{tikzcd}
		\end{center}
Any two maximal tori are conjugate hence induce the same map in $H^*_G(B_GH)$ so we get
		\begin{center}
			\begin{tikzcd}
			A_{\Q}[\genw_1,...,\genw_n,\genu_1,...,\genu_n]/x\genu_i&A_{\Q}[\genw_1,...,\genw_{2n},\genu_1,...,\genu_{2n}]/x\genu_i\ar[l]\\
			H^*_GSO(2n)\ar[u,hook]&H^*_GU(2n)\ar[u,hook]\ar[l]
			\end{tikzcd}
		\end{center}
		In the top row, $\genw_i\mapsto \genw_i, \genw_{n+i}\mapsto -\genw_i, \genu_i\mapsto \genu_i, \genu_{n+i}\mapsto \genu_i$ for $i\le n$. This implies the formulas for the effect of the complexification map.	\end{proof}
	
	We only have one map $B_GSO(n)\to B_GSO(n+1)$, given by direct sum with the trivial representation; direct sum with the $\sigma$ representation does not result in a $C_2$ equivariant map. There is however a $C_2$-equivariant map $B_GSO(n)\to B_GSO(n+2)$ given by adding $2\sigma$.
	
	The map $B_GSO(n)\xrightarrow{\oplus 1} B_GSO(n+1)$ induces
	\begin{gather}
	\genu\mapsto y+\genu\\
	p_i\mapsto p_i\\
	\pi_{s,j}\mapsto \pi_{s,j}+\pi_{s,j-1}\\
	\chi\mapsto 0
	\end{gather}
	(with the usual conventions on the RHS, including that $\pi_{s,0}=yp_s$). 
	
	The map $B_GSO(n)\xrightarrow{\oplus 2\sigma} B_GSO(n+2)$ induces:
	\begin{gather}
	\genu\mapsto \genu\\
	p_i\mapsto p_i\\
	\pi_{s,i}\mapsto \pi_{s,i}\\
	\chi\mapsto 0
	\end{gather}
	(with the usual conventions on the RHS).

	So we can distinguish between $B_G^+SO, B_G^-SO$ and $B_G^{\pm}SO$ as usual. The results here are then parallel to the $B_G^+U, B_G^-U$ and $B_G^{\pm}U$ cases respectively, by replacing $\genu,\gene_i,\genc_{s,i}$ by $\genu,p_i,\pi_{s,i}$; the Euler class $\chi$ is not stable. See subsection \ref{StableSummary} for a summary.

	\section{Orthogonal groups}\label{Orthogonal}
	
	The groups $O(2n), SO(2n+1), O(2n+1)$ have the same maximal torus and Weyl group. The resulting nonequivariant classifying spaces have isomorphic cohomology rings and the maximal torus isomorphism works in all cases (even though the orthogonal groups are disconnected). In this subsection, we shall see that this observation doesn't generalize to the $C_2$ equivariant case.
	
	For $B_GO(n)$ we compute the fixed points:
	\begin{equation}
	B_GO(n)^{C_2}=\coprod_{m+k=n}BO(m)\times BO(k)
	\end{equation}
	%	Comparing said fixed points we can see that for no $n$ are any two of $H^{\bigstar}(B_GO(2n)), H^{\bigstar}(B_GO(2n+1)), H^{\bigstar}(B_GSO(2n+1))$ ever isomorphic, contrary to the nonequivariant case. 
	
	We consider the map
	\begin{equation}
	H^*_G(B_GO(n))\to H^*_G(B_GT)^W
	\end{equation}
	where $T\subseteq O(n)$ is a maximal torus and $W$ the Weyl group. We can see directly that for $n\le 3$ the RHS has smaller dimension compared to the LHS even before taking fixed points. So this map cannot be an isomorphism for $n\le 3$.
	
	We can also see from the fixed point computation that $H^*_G(B_GO(n))$ is never isomorphic to $H^*_G(B_GO(n+1))$ for any $n$.\medbreak
	
	There is however a natural description of the characteristic classes for the group $O(2n+1)$: since $O(2n+1)=SO(2n+1)\times O(1)$ we have that
	\begin{equation}
		B_GO(2n+1)=B_GSO(2n+1)\times B_GO(1)
	\end{equation}
hence 
	\begin{equation}
	H^*_G(B_GO(2n+1))=H^*_G(B_GSO(2n+1))\boxtimes_{A_{\Q}} H^*_G(B_GO(1))
\end{equation}
and
	\begin{equation}
H^*_G(B_GO(1))=A_{\Q}[\gend]/(\gend^2=\gend, x\gend)
\end{equation}
for $\gend$ in degree $0$. So the $C_2$-characteristic classes for $O(2n+1)$ are $\genu, p_i, \pi_{s,j}, \gend$
	
	The map $B_GO(2n)\to B_GO(2n+1)$ is always a surjection in cohomology (this follows by looking at the fixed points) so $H^*_G(B_GO(2n))$ is a quotient of $H^*_G(B_GO(2n+1))$. For $n=1$ we get: 
	\begin{gather}
H^*_G(B_GO(2))=\frac{A_{\Q}[\genu,\gend,p_1]}{x\gend, \genu^2=\genu, \gend^2=\gend, \genu\gend=\genu, yp_1=\gend p_1}
\end{gather}
More generally, in degree $0$:
\begin{equation}
H^0_G(B_GO(2n))=H^0_G(B_GO(2n+1))/\genu(\genu-1)\cdots (\genu-n+1)(1-\gend)
\end{equation}
For $n=2$ the other relations in higher degrees are $p_1\genu^2(1-\gend)=0$ and $p_2(y-\gend-\genu+\genu\gend)=0$. 

Finally, stability for orthogonal groups is understood from $B_GO=B_GSO\times B_GO(1)$.

	\section{Special unitary groups}\label{SU}
	
	\begin{prop}	The maximal torus inclusion $U(1)^{n-1}\to SU(n)$ induces an isomorphism 
		\begin{equation}
		H^*_G(B_GSU(n))\to H^*_G(B_GU(1)^{n-1})^{\Sigma_n}
		\end{equation}
		
	\end{prop}
	
	\begin{proof}

		%	For $B_GSU(n)$ we use the model $SU/(SU(n)\times SU(\infty-n))$ i.e. $n$-dimensional orthonormal frames in $\C^{\infty\rho}$ mod $(v_1,...,v_n)\sim A(v_1,...,v_n)$ where $A\in SU(n)$ (and $A(v_1,...,v_n)$ has $i$-th vector $a_{i1}v_1+\cdots+a_{in}v_n$). Any frame fixed by the $C_2$ action can be decomposed into $W^+\oplus W^-$ where $\dim(W^-)$ must be even; this is due to $m\sigma$ being an $SU(m)$ rep iff $m$ is even. And if $W^+\oplus W^-\approx V^+\oplus V^-$ through an $SU(n)$ matrix $A$, then $A$ is block diagonal with blocks $B\in U(n-2k), C\in U(2k)$ and $\det(B)\det(C)=1$; $B$ gives $W^+\approx V^+$ while $C$ gives $W^-\approx V^-$. Thus
		
		Analogously to $B_GSO(n)$,
		\begin{gather}
		B_GSU(n)^{C_2}=\coprod_{k=0}^{n/2}BZ'_{2k,n}
		\end{gather}
		where $Z'_{k,n}\subseteq U(k)\times U(n-k)$ consists of $(A,B)$ with $\det(A)\det(B)=1$. So $Z'_{k,n}=Z'_{n-k,n}$ and $Z'_{0,n}=SU(n)$. If $0<k<n$,
		\begin{equation}
		Z'_{k,n}=(SU(k)\times SU(n-k))\rtimes S^1
		\end{equation}
		where $S^1$ acts diagonally by conjugation. The maximal torus in $Z'_{2k,n}$ is $U(1)^{n-1}$ and the Weyl group is $\Sigma_{2k}\times \Sigma_{n-2k}$.
		
		%In particular $B_GSU(2)^{C_2}=BSU(2)\coprod BSU(2)$ hence $\dim H^0(B_GSU(2)^{C_2})=2$. On the other hand, we have $H^0(B_GU(1)^{C_2})=\Q\times \Q$ with $\Sigma_2$ action $(1,0)\mapsto(0,1)$ so the dimension of the fixed points is $1$. Thus
		%\begin{equation}
		%H^{\bigstar}(B_GSU(n))\neq (\boxtimes_{n-1}H^{\bigstar}BU(1))^{\Sigma_n}
		%\end{equation}
		%through the maximal torus inclusion.
		
		Taking fixed points in $B_GU(1)^{n-1}\to B_GSU(n)$ gives:
		\begin{equation}
		\coprod^{2^{n-1}}BU(1)^{n-1}\to \coprod_{k=0}^{n/2}BZ'_{2k,n}
		\end{equation}
		%		sending $(v_1^{\pm},...,v_{n-1}^{\pm})\mapsto \epsilon^{\pm}\oplus v_1^{\pm}\oplus \cdots \oplus v_{n-1}^{\pm}$ in $\C^{\rho}\oplus \C^{\infty\rho}$ where $\epsilon^+=1$ and $\epsilon^-=\sigma$ are two reserved vectors in $\C^{\rho}\subseteq \C^{\rho}\oplus \C^{\infty \rho}$; the sign of $\epsilon^{\pm}$ is chosen depending on the number of $-$'s in the sign configuration, so that the total number of $-$ signs (after adding $\epsilon^{\pm}$) is even. 
		and as in $B_GSO(n)$, the coproduct breaks into
		\begin{equation}
		\coprod^{\binom{n}{2k}}BU(1)^{n-1}\to BZ'_{2k,n}
		\end{equation}
		%where $2k$ is the number of $-$'s in $\epsilon^{\pm}\oplus v_1^{\pm}\oplus \cdots \oplus v_{n-1}^{\pm}$. 
		In cohomology:
		\begin{equation}
		H^*(BZ'_{2k,n})\to \oplus^{\binom{n}{2k}}H^*(BU(1))^{\otimes (n-1)}
		\end{equation}
	The group	$\Sigma_n$ acts on the right by permuting the sign configuration and the tensor factors $H^*(BU(1))=\Q[a_i]$ where $a_n=-(a_1+\cdots+a_{n-1})$.	Thus, taking $\Sigma_n$ fixed points is equivalent to fixing a sign configuration, say $(-,...,-,+,...,+)$, and then taking $\Sigma_{2k}\times \Sigma_{n-2k}$ fixed points, which is exactly the Weyl group of $Z'_{2k,n}$. This establishes the maximal torus isomorphism for $B_GSU(n)$.
	\end{proof}
	
	In conclusion, $H^*_G(B_GSU(n))=R'^{\Sigma_n}$ where $R'=A_{\Q}[\genw_i,\genu_i]_{1\le i\le n}$ modulo the relations \begin{gather}
	\genu_n= \frac{(-1)^{n+1}+1}2y+(-1)^n\sum_{i=1}^{n-1}(-2)^{i-1}\sigma_i(\genu_1,...,\genu_{n-1})\\
	\genw_n=-\genw_1\cdots-\genw_{n-1}
	\end{gather}
	Using the same definitions for $\genu,\gene_i,\genc_{s,j}$ in terms of $\genu_i,\genw_i$ as in $B_GU(n)$, we can see that
	\begin{equation}
	\gene_1=\genc_{1,n-1}=0
	\end{equation}
	There are more relations however: for example, $SU(2)=Sp(1)$ so we need: 
	\begin{equation}
	\genu^2=2\genu
	\end{equation}
	as an additional relation. The identification $SU(2)=Sp(1)$ then becomes:
	\begin{gather*}
	\frac{A_{\Q}[\gene_2,\genu]}{\genu^2=2\genu, x\genu}\to \frac{A_{\Q}[k_1,\genu]}{\genu^2=\genu}\\
	\gene_2\mapsto k_1\\
	\genu\mapsto 2\genu
	\end{gather*} The rest of the relations for each $B_GSU(n)$ can be computed algorithmically. 
	%It's also possible that fewer generators are required for $B_GSU(n)$ than the set $\genu,\gene_i,\genc_{s,j}$ for $i\neq 1, (s,j)\neq (1,n-1)$.

	As for stability, we have the map $B_GSU(n)\to B_GSU(n+1)$ given by direct sum with a trivial representation. The map $B_GSU(n)\to B_GSU(n+1)$ given by direct sum with a $\sigma$ representation is not $C_2$ equivariant, so we instead use the map $B_GSU(n)\to B_GSU(n+2)$ adding $2\sigma$.

	Analogously to the $B_GU$ case, we can distinguish between spaces $B_G^+SU, B_G^-SU$ and $B_G^{\pm}SU$. Under the inclusion $B_G^{\pm}SU\to B_G^{\pm}U$, $\gene_1=\genc_1=0$.% the sub-Hopf-algebra of $H^*_G(B_G^{\pm}U)$ of constant sequences
	%\begin{equation}\frac{A_{\Q}[\gene_i,\genc_i]_{i\ge 1}}{x\genc_i}
	%\end{equation}
	%projects to a sub-Hopf-algebra of $H^*_G(B_G^{\pm}SU)$ with additional relations given by $\gene_1=\genc_1=0$. 
	%constant sequences
	%and $B_G^sSU$.
	%\begin{equation*}
	%B_G^sSU^{C_2}=BZ'_{\infty,\infty}\times \Z'
	%\end{equation*}
	%where $\Z'_{\infty,\infty}=(SU\times SU)\rtimes S^1$ is the subgroup of $U\times U$ of pairs $(A,B)$ with $\det(A)\det(B)=1$.

\appendix

\section{Symmetric Polynomials with Relations}\label{appen}

For a fixed commutative ring $k$ consider the graded $k$-algebra $R$:
\begin{equation}
R=R(\genw_1,...,\genw_n,\genu_1,...,\genu_n):=k[\genw_i,\genu_i]/(\genu_i^2=\genu_i)
\end{equation}
whose generators have degrees $|\genw_i|=1$ and $|\genu_i|=0$. The group $\Sigma_n$ acts on $R$ by permuting the $\genw_i$ and $\genu_i$ separately.\medbreak

Any monic monomial in $R$ takes the unique form
\begin{equation}
\genw_1^{a_1}\cdots \genw_n^{a_n}\genu_1^{\epsilon_1}\cdots \genu_n^{\epsilon_n}
\end{equation}
for $a_i\ge 0$ and $\epsilon_i=0,1$. We order the monic monomials of the same degree lexicographically by the powers $a_1,...,a_n,\epsilon_1,...,\epsilon_n$. For a homogeneous element $p\in R$, we consider the (nonzero) monomials $p_i$ in $p$ and let $p'_i$ be the corresponding monic monomials; the greatest of the $p'_i$ is the dominant term $\dom(p)$ of $p$. We compare homogeneous polynomials using their dominant terms (ignoring their coefficients).

\begin{prop}\label{AlgebraPropositionGeneral}We have the $k$-algebra presentation:
	\begin{equation}
	R^{\Sigma_n}=\frac{k[\genc_{s,i}]}{\genc_{s,i}\genc_{t,j}=\binom{\min(i+j+s,n)-t}{j}\genc_{t,0}\genc_{s,\min(i+j,n-s)}+p_{s,i,t,j,n}}
	\end{equation}
	where:
	\begin{itemize}
		\item For each pair of nonnegative indices $s,i$ such that $s+i\le n$ we have one generator $\genc_{s,i}$ of degree $s$ given by:
		\begin{gather}
		\genc_{s,i}=\sum_{(m_*,l_*)}\genw_{m_1}\cdots \genw_{m_s}\genu_{l_1}\cdots \genu_{l_i}
		\end{gather}
		where $(m_*,l_*)$ ranges over pairs of disjoint partitions $1\le m_1<\cdots<m_s\le n$ and $1\le l_1<\cdots<l_i\le n$. 
		
		\item For each quadruple of indices $s,i,t,j$ such that $0\le s\le t\le s+i$, $0<i\le n-s$ and $0<j\le n-t$, we have a relation:
		\begin{gather}
		\genc_{s,i}\genc_{t,j}=\binom{\min(i+j+s,n)-t}{j}\genc_{t,0}\genc_{s,\min(i+j,n-s)}+p_{s,i,t,j,n}
		\end{gather}
		where $p_{s,i,t,j,n}$ is a homogeneous polynomial in $R^{\Sigma_n}$ smaller than $\genc_{t,0}\genc_{s,\min(i+j,n-s)}$.
		
	\end{itemize}
	The elements $\gene_i=\genc_{i,0}$ are algebraically independent over $k$ and $R^{\Sigma_n}$ is finite over $k[\gene_1,...,\gene_n]$ so we have the equality of Krull dimensions:
	\begin{equation}
	\dim(R)=\dim(k)+n
	\end{equation}
	A $k$-module basis of $R^{\Sigma_n}$ consists of the elements
	\begin{equation}
	\prod_{i=1}^n\gene_i^{r_i}\prod_{s=0}^n\prod_{i=1}^{n-s} \genc_{s,i}^{\epsilon_{s,i}}
	\end{equation}
	where the $\epsilon_{s,i}=0,1$ are such that whenever $\epsilon_{s,i}=\epsilon_{t,j}=1$ for $1\le i,j$ and $s\le t$ then we must also have $s+i<t$.
\end{prop}

Three remarks about the polynomials $p_{n,s,i,t,j}$ appearing in the relations:

\begin{itemize}
	\item The proof of Proposition \ref{AlgebraPropositionGeneral} provides an algorithm for computing these polynomials. This algorithm has been implemented in a computer program available \href{https://github.com/NickG-Math/Symmetric_Polynomials}{here} (executable files are available \href{https://github.com/NickG-Math/Symmetric_Polynomials/releases}{here} for a quick demonstration).
	\item The coefficients of the polynomials $p_{s,i,t,j,n}$ are in the image of the initial homomorphism $\Z\to k$ (i.e. they are independent of $k$).
	\item If $n\ge i+j+s$ then $p_{s,i,t,j,n}$ is independent of $n$ and moreover the relation on $\genc_{s,i}\genc_{t,j}$  is independent of $n$, taking the simpler form:
	\begin{gather}
	\genc_{s,i}\genc_{t,j}=\binom{i+j+s-t}{j}\gene_t\genc_{s,i+j}+p_{s,i,t,j}
	\end{gather}
\end{itemize}

The elements $\gene_i=\genc_{i,0}$ generating $k[\gene_1,...,\gene_n]$ are the elementary symmetric polynomials on the variables $\genw_i$:
\begin{gather}
\gene_i=\sigma_i(\genw_1,...,\genw_n)=\sum_{1\le m_1<\cdots<m_i\le n}\genw_{m_1}\cdots \genw_{m_i}
\end{gather}
Similarly, the elements $\genc_{0,i}$ generating $R^{\Sigma_n}_0$ are the elementary symmetric polynomials on the variables $\genu_i$:
\begin{gather}
\genc_{0,i}=\sigma_i(\genu_1,...,\genu_n)=\sum_{1\le l_1<\cdots<l_i\le n}\genu_{l_1}\cdots \genu_{l_i}
\end{gather}

\begin{prop}\label{AlgebraProposition2General}Any set of homogeneous algebra generators of $R^{\Sigma_n}$ over $R^{\Sigma_n}_0$ has cardinality at least $n+\binom n2$, which is the cardinality of the generating set $\{\gene_s,\genc_{s,j}\}_{s,j>0}$.
\end{prop}

A minimal generating set for the $k$-algebra $R^{\Sigma_n}_0$ depends on which primes are invertible in $k$. For example, if $k$ is a $\Q$-algebra, which is our case of interest, we can generate all $\genc_{0,i}$ from the single element
\begin{equation}
\genu=\genc_{0,1}=\genu_1+\cdots+\genu_n
\end{equation}
via the formula:
\begin{equation}
\genc_{0,i}=\frac{\genu(\genu-1)\cdots (\genu-i+1)}{i!}
\end{equation}
The disadvantage of supplanting $\genc_{0,i}$ with $\genu$ is that the relations between the generators now require additional rational coefficients, as in Corollary \ref{AlgebraCorollaryQ}. For the purposes of the algorithm implemented in our \href{https://github.com/NickG-Math/Symmetric_Polynomials}{computer program}, it is better (in terms of speed and numerical stability) to use all the $\genc_{0,i}$ and $\Z$-coefficients.\medbreak

The rest of this appendix is dedicated to proving Propositions \ref{AlgebraPropositionGeneral} and \ref{AlgebraProposition2General}.

\begin{proof}\label{ProofAlgorithmGeneral}(Of Proposition \ref{AlgebraPropositionGeneral}) We will need that the dominant term of $\genc_{s,i}$ is:
	\begin{gather}
	\dom(\genc_{s,i})=\genw_1\cdots \genw_s\genu_{s+1}\cdots \genu_{s+i}
	\end{gather}	
	Let us note a subtlety about dominant terms and multiplication: $\dom(pq)$ is $\dom(p)\dom(q)$ if $p$ or $q$ are polynomials solely on the $\genw_i$, but if both $p,q$ contain $\genu_i$'s that may not be the case: for example $\genc_{1,1}$ has dominant term $\genw_1\genu_2$ but $\genc_{1,1}^2$ has dominant term $\genw_1^2\genu_2\genu_3$.

	Every $\Sigma_n$ orbit of a monic monomial in $R$ has a greatest term $M$ that can be written as either:
	\begin{equation}
	M=\genw_1^{a_1}\cdots \genw_s^{a_s}\genu_1^{\epsilon_1}\cdots \genu_s^{\epsilon_s}\genu_{s+1}\cdots \genu_{s+i}
	\end{equation}
	with $a_1\ge \cdots\ge a_s> 0$ and $\epsilon_i=0,1$, or as:
	\begin{equation}
	M=\genu_1\cdots\genu_k
	\end{equation}
	It suffices to prove that any such $M$ is the dominant term of a polynomial on $\genc_{s,j}$. Note that $M=\genu_1\cdots\genu_k$ is the dominant term of $\genc_{0,k}$, so we may restrict our attention exclusively to $M$'s of the first form. It should also be noted that we can't assume that the $\epsilon_i$ are in decreasing order, since applying a permutation to fix such an order would affect the decreasing order on the $a_i$. 
	
	So let
	\begin{equation}
	M=\genw_1^{a_1}\cdots \genw_s^{a_s}\genu_1^{\epsilon_1}\cdots \genu_s^{\epsilon_s}\genu_{s+1}\cdots \genu_{s+i}
	\end{equation}
	be greatest in its $\Sigma_n$ orbit, where $a_1\ge \cdots\ge a_s> 0$ and $\epsilon_i=0,1$. We shall prove that $M=\dom(P)$ where $P$ is a product of $\genc_{s,j}$. To ensure that $P$ is unique per $M$, we insist that if $\genc_{t,j}, \genc_{t',j'}$ are factors of $P$ with $0<j,j'$ and $t\le t'$ then $t+j<t'$. With this extra requirement, no two distinct products $P$ can have the same dominant term. 
	
	Further simplifying matters, note that it suffices to write
	\begin{equation}
	M=\genw_1^{k_1}\cdots \genw_s^{k_s}\dom(P')
	\end{equation}
	where $k_1\ge\cdots\ge k_s\ge 0$ and $P'$ is a product of $\genc_{s,i}$ with $i>0$ satisfying the condition above. This is because $\genw_1^{k_1}\cdots \genw_s^{k_s}$ is the dominant term of a product $P_{\gene}$ of $\gene_i$ by the fundamental result on symmetric polynomials. Then we can take $P=P_{\gene}P'$ and $M=\dom(P)$.
	
	We distinguish cases on the number of $\epsilon_i$'s in $M$ that are nonzero, i.e. the number of $\genu_i$'s in $M$ with $i\le s$. If there are none then
	\begin{equation}
	M=	\genw_1^{a_1}\cdots \genw_s^{a_s}\genu_{s+1}\cdots \genu_{s+i}
	\end{equation}
	can be written as
	\begin{equation}
	(\genw_1^{a_1-1}\cdots \genw_s^{a_s-1})(\genw_1\cdots \genw_s\genu_{s+1}\cdots \genu_{s+i})
	\end{equation} 
	so we use $P'=\genc_{s,i}$.\\
	Now assume there's only one $\genu_j$ with $j\le s$:
	\begin{equation}
	M=	\genw_1^{a_1}\cdots \genw_s^{a_s}\genu_j\genu_{s+1}\cdots \genu_{s+i}
	\end{equation}
	If $j>1$ notice that $a_{j-1}>a_j$ for otherwise we can exchange $j-1,j$ and get a greater term in our order, contradicting that $M$ is greatest in its $\Sigma_n$ orbit. If further $j<s$, we can write $M$ as
	\begin{equation}
	(\genw_1^{a_1-2}\cdots \genw_{j-1}^{a_{j-1}-2}\genw_j^{a_j-1}\cdots \genw_s^{a_s-1})(\genw_1\cdots \genw_{j-1}\genu_j)(\genw_1\cdots \genw_s\genu_{s+1}\cdots \genu_{s+i})
	\end{equation}
	and use  $P_{\genu,\genc}=\genc_{j-1,1}\genc_{s,i}$. If $j=s$ we instead write $M$ as
	\begin{equation}
	(\genw_1^{a_1-1}\cdots \genw_{s-1}^{a_{s-1}-1}\genw_s^{a_s})(\genw_1\cdots \genw_{s-1}\genu_s\genu_{s+1}\cdots \genu_{s+i})
	\end{equation}
	and use $P'=\genc_{s-1,i+1}$. If $j=1$ then 
	\begin{equation}
	M=	\genw_1^{a_1}\cdots \genw_s^{a_s}\genu_1\genu_{s+1}\cdots \genu_{s+i}
	\end{equation}
	is 
	\begin{equation}
	(\genw_1^{a_1-1}\cdots \genw_s^{a_s-1})\genu_1(\genw_1\cdots \genw_s\genu_{s+1}\cdots \genu_{s+i})
	\end{equation} 
	and we use $P'=\genc_{0,1}\genc_{s+1,i}$.
	
	Now assume there are two $\genu_j$'s with $j\le s$ in $M$, say $\genu_j,\genu_k$ with $j<k\le s$:
	\begin{equation}
	M=\genw_1^{a_1}\cdots \genw_s^{a_s}\genu_{j}\genu_k\genu_{s+1}\cdots \genu_{s+i}
	\end{equation}
	If $j>1$ then as before we must have $a_{j-1}>a_j$ and $a_{k-1}>a_k$. If furthermore $j<k-1$ and $k<s$ we can write $M$ as
	\begin{align}
	\left(\prod_{l=1}^{j-1}\genw_l^{a_l-3}\prod_{l=j}^{k-1}\genw_l^{a_l-2}\prod_{l=k}^{s}\genw_l^{a_l-1}\right)(\genw_1\cdots \genw_{j-1}\genu_j)(\genw_1\cdots \genw_{k-1}\genu_k)(\genw_1\cdots \genw_s\genu_{s+1}\cdots \genu_{s+i})
	\end{align}
	and use $P'=\genc_{j-1,1}\genc_{k-1,1}\genc_{s,i}$. If $j=k-1$ and $k<s$ then write $M$ as
	\begin{align}
	(\genw_1^{a_1-2}\cdots \genw_{j-1}^{a_{j-1}-2}\genw_j^{a_j-1}\cdots \genw_s^{a_s-1})(\genw_1\cdots \genw_{j-1}\genu_j\genu_{j+1})(\genw_1\cdots \genw_s\genu_{s+1}\cdots \genu_{s+i})
	\end{align}
	and use $P'=\genc_{j-1,2}\genc_{s,i}$. The other cases are all handled similarly. Note that for $j=1,k=2$ we get
	\begin{equation}
	(\genw_1^{a_1-1}\cdots \genw_s^{a_s-1})(\genu_1\genu_2)(\genw_1\cdots \genw_s\genu_{s+1}\cdots \genu_{s+i})
	\end{equation}
	and use $P'=\genc_{0,2}\genc_{s,i}$.
	
	We proceed in this fashion to treat the case where $r$ many $\genu_j$'s with $j\le s$ appear in $M$, for any $r=0,...,s$.
\end{proof}

We now prove Proposition \ref{AlgebraProposition2General} regarding the minimality of generators.
\begin{proof}(Of Proposition \ref{AlgebraProposition2General}) For every $s$ with $1\le s\le n$, consider the equivalence relation on $R^{\Sigma_n}_s$ whereby two elements are equivalent if they have equal dominant terms. The orbit set of the equivalence relation, when ordered from least to greatest, starts with $\genc_{s,0},\genc_{s,1},...,\genc_{s,n-s}$ and continues with products of elements in $R^{\Sigma_n}_t$, $t<s$.	

Let $X$ be a homogeneous generating set of $R^{\Sigma_n}$ over $R^{\Sigma_n}_0$ of minimum cardinality, ordered first by degree and then by dominant term. The smallest element $x\in X$ must then satisfy $\dom(x)=\dom(\genc_{1,0})$ hence $x-r\genc_{1,0}\in R^{\Sigma_n}_0$ for some $r\in k$; this means that we can replace $x$ by $\genc_{1,0}$ resulting in a new generating set $X$ with minimum cardinality. Applying the same argument repeatedly shows that all elements of $X$ can be replaced by $\genc_{1,0},\genc_{1,1},...,\genc_{n,0}$ while preserving cardinality and polynomial span.\end{proof}

\phantom{1}\smallbreak

\begin{small}
	\noindent  \textsc{Department of Mathematics, University of Chicago}\\
	\textit{E-mail:} \verb|nickg@math.uchicago.edu|\\
	\textit{Website:} \href{http:://math.uchicago.edu/~nickg}{math.uchicago.edu/$\sim$nickg}
\end{small}


\begin{thebibliography}{999}
	
	\bibitem[BCM]{BCM} Robert R. Bruner, Michael J. Catanzaro, J. P. May, \emph{Characteristic classes},	Book in preparation, available \href{https://www.math.uchicago.edu/~may/CHAR/charclasses.pdf}{here}
	\smallbreak
	
	\bibitem[Cho18]{Cho18}	Z. Chonoles, \emph{The $RO(G)$-graded cohomology of the equivariant classifying space $B_GSU(2)$}, \href{https://arxiv.org/abs/1808.00604}{https://arxiv.org/pdf/1808.00604} 
	\smallbreak			
			
	\bibitem[Geo19]{Geo19} 	N. Georgakopoulos, \emph{The $RO(C_4)$ integral homology of a point},	\href{https://arxiv.org/pdf/1912.06758.pdf}{arXiv:1912.06758}
	\smallbreak
		
	\bibitem[Geo21a]{BC4S2} 	N. Georgakopoulos, \emph{The $RO(C_4)$ cohomology of the infinite real projective space}, available	\href{https://math.uchicago.edu/~nickg/papers/BC4Sigma2.pdf}{here}
	\smallbreak
	
	\bibitem[GM95]{GM95} J.P.C. Greenlees, J.P. May, \emph{Generalized Tate Cohomology}, Memoirs of the American Mathematical Society 543 (1995).	
	\smallbreak	
	
	\bibitem[HK96]{HK96} P. Hu and I. Kriz, \emph{	Real-oriented homotopy theory and an analogue of the Adams Novikov spectral sequence},	 Topology 40 (2001), no. 2, 317–399
	\smallbreak	
		
	\bibitem[Lew88]{Lew88} L. G. Lewis, \emph{The RO(G)-graded equivariant ordinary cohomology of
	complex projective spaces with linear $\Z/p$ actions},	Algebraic topology and transformation groups, Volume 1361, pg 53–122, Springer, 1988.
	\smallbreak
	
	\bibitem[May87]{May87} J.P. May, \emph{Characteristic Classes in Borel Cohomology},	Journal of Pure and Applied Algebra 44 (1987) 287-289, North-Holland
	\smallbreak
	
	\bibitem[Shu14]{Shu14}	M. Schulman, \emph{Equivariant local coefficients and the $RO(G)$-graded cohomology of classifying spaces }, \href{https://arxiv.org/abs/1405.1770}{arXiv:1405.1770} 
	\smallbreak
	
	\bibitem[Wil19]{Wil19} D. Wilson, \emph{$C_2$-equivariant Homology Operations: Results and Formulas},	\href{	https://arxiv.org/pdf/1905.00058.pdf}{arXiv:1905.00058}
	\smallbreak
	
	\bibitem[Zeng17]{Zeng} M. Zeng, \emph{Equivariant Eilenberg-MacLane spectra in cyclic $p$-groups
	}, \href{https://arxiv.org/pdf/1710.01769}{arXiv:1710.01769v2}
	
	
\end{thebibliography}
\end{document}